\documentclass[11pt]{article}
\usepackage[utf8]{inputenc}
\usepackage[english]{babel}
\usepackage{amsmath, amssymb, theorem, latexsym, epsfig, eufrak}
\usepackage{subfigure}
\numberwithin{equation}{section}

\theoremstyle{plain}
\theorembodyfont{\itshape}
\newtheorem{theorem}{Theorem}[section]
\newtheorem{proposition}[theorem]{Proposition}
\newtheorem{lemma}[theorem]{Lemma}
\newtheorem{corollary}[theorem]{Corollary}
                                                                                
\theorembodyfont{\rmfamily}
\newtheorem{definition}[theorem]{Definition}

\newtheorem{remark}[theorem]{Remark}

\newtheorem{conjecture}[theorem]{Conjecture}
\newenvironment{proof}{{\noindent \textbf{Proof}\,\,}}{\hspace*{\fill}$\Box$\medskip}

\usepackage{graphicx}
\usepackage{hyperref}
\usepackage{pdfpages}
\hypersetup{
    colorlinks,
    linkcolor={black!50!black},
    citecolor={green!50!black},
    urlcolor={blue!80!black}
}
\usepackage{upgreek}
\usepackage{MnSymbol}
\usepackage{mathrsfs}
\usepackage{xcolor}
\usepackage{listings}
\usepackage{caption}
\usepackage{subfigure}
\usepackage{xypic}
\usepackage[all]{xy}

\definecolor{string}{HTML}{B40000} 
\definecolor{comment}{HTML}{008000} 
\definecolor{keyword}{HTML}{1A00FF} 
\definecolor{morecomment}{HTML}{8000FF}
\definecolor{bk}{HTML}{FFFFFF} 
\definecolor{frame}{HTML}{999999} 
\definecolor{brackets}{HTML}{B40000} 
 
\DeclareCaptionFont{white}{\color{??aptiontext}}
\DeclareCaptionFormat{listing}{\parbox{\linewidth}{\colorbox{??aptionbk}{\parbox{\linewidth}{#1#2#3}}\vskip-4pt}}
\lstdefinelanguage{XML}
{
	morestring=[s]{"}{"},
	morecomment=[s]{?}{?},
	morecomment=[s]{!--}{--},
	commentstyle=\color{comment},
	moredelim=[s][\color{black}]{>}{<},
	moredelim=[s][\color{red}]{\ }{=},
	stringstyle=\color{string},
	identifierstyle=\color{keyword}
}

\title{On spectral curves and complexified boundaries of the phase-lock areas in  
a model of Josephson junction}
\author{
A.\,A.\,Glutsyuk
\thanks{ CNRS, France (UMR 5669 (UMPA, ENS de Lyon) and UMI 2615 (Lab. J.-V.Poncelet)), Lyon, France. 
Email: aglutsyu@ens-lyon.fr}
\thanks{National Research University Higher School of Economics (HSE), Moscow, Russia}
\thanks{The results of Sections 1.1, 1.5-1.7, 2, 4, 5 are obtained by A.\,A.\,Glutsyuk. His research was supported by RSF grant 18-41-05003.}, 
I.\,V.\,Netay
\thanks{Institute for Information Transmision Problems, Russian Academy of Sciences}
\thanks{Stock Research and production company Kryptonite}
\thanks{
The study has been funded within the framework of the HSE University Basic Research Program and the Russian Academic Excellence Project '5-100.
}
\thanks{The results of Sections 1.2, 1.3,  3 are obtained by I.\,V.\,Netay}
}

\begin{document}
\maketitle
\def\la{\lambda}
\def\La{\Lambda}
\def\cc{\mathbb C}
\def\oc{\overline\cc}
\def\eps{\varepsilon}
\def\rr{\mathbb R}
\def\zz{\mathbb Z}
\def\cp{\mathbb{CP}}
\def\diag{\operatorname{diag}}
\def\mcl{\mathcal L}
\def\nn{\mathbb N}
\def\var{\varepsilon}
\def\mcr{\mathcal R}
\def\Simple{\operatorname{SI}}
\def\si{\Simple}
\def\hsi{\widehat{\Simple}}
\def\hmcl{\widehat{\mcl}}
\def\wt#1{\widetilde#1}
\def\sign{\operatorname{sign}}
\def\mca{\mathcal A}
\def\rp{\mathbb{RP}}
\def\mcp{\mathcal P}
\def\modulo{\operatorname{mod}}

\begin{abstract}
    The paper deals with a three-parameter 
    family of special double confluent Heun equations that was introduced and studied by 
    V.\,M.\,Buchstaber and S.\,I.\,Tertychnyi  as an equivalent presentation of  
    a model of overdamped Josephson junction in superconductivity. The 
    parameters are $l,\la,\mu\in\rr$. Buchstaber and Tertychnyi described those parameter values, for which the corresponding equation has 
    a polynomial solution. They have shown that for $\mu\neq0$ this happens exactly when 
    $l\in\nn$ and the parameters $(\la,\mu)$ lie on an algebraic curve $\Gamma_l\subset\cc^2_{(\la,\mu)}$ called 
    the {\it $l$-spectral curve} and defined as zero locus of determinant of 
    a remarkable three-diagonal $l\times l$-matrix. They studied the real part of the 
    spectral curve and obtained important results with applications to model of 
    Josephson junction, which is a family of dynamical systems on 2-torus depending on 
    real parameters $(B,A;\omega)$; the parameter $\omega$, called the frequency, 
    is fixed. One of main problems on the above-mentioned model 
    is to study the geometry of boundaries of its phase-lock areas in $\rr^2_{(B,A)}$ 
     and their evolution, as $\omega$ decreases to 0. An approach 
    to this problem suggested in the present paper is to study the complexified boundaries. 
    We prove irreducibility of the complex spectral curve $\Gamma_l$ 
    for every $l\in\nn$. We also calculate its genus for $l\leqslant20$ and present a conjecture on 
    general genus formula. 
    We apply the  irreducibility result to the complexified boundaries of the phase-lock areas of model of Josephson junction. 
    The family of  real boundaries  taken for all $\omega>0$ yields a {\it countable} union 
    of two-dimensional analytic surfaces in $\rr^3_{(B,A,\omega^{-1})}$. 
    We show that, unexpectedly, its complexification 
    is a complex analytic subset  consisting of just {\it four} two-dimensional 
    irreducible components, and we describe them. This is done by using the representation 
    of some special points of the boundaries (the so-called generalized simple intersections) as points of the real spectral curves and the above irreducibility 
    result.     We also prove that the spectral curve 
    has no real ovals. We present a Monotonicity  Conjecture on the evolution 
    of the phase-lock area portraits, as $\omega$ decreases, and a partial positive 
    result towards its confirmation. 
\end{abstract}

\tableofcontents

\section{Introduction and main results}

The paper deals with the family of special double confluent Heun equations 
\begin{equation} 
    z^2E''+((-l+1)z+\mu(1-z^2))E'+(\lambda+\mu(l-1)z)E=0; \ 
    l,\la,\mu\in\cc.
    \label{heun}
\end{equation}
The above family, which belong to the well-known class of Heun equations, 
see \cite{sl}, was studied by V.\,M.\,Buchstaber and S.\,I.\,Tertychnyi in \cite{tert2, 
bt0, bt1, bt3}. They have shown in \cite{bt0, bt1} that its restriction to real parameters satisfying the inequality $\la+\mu^2>0$ is equivalent to 
a model of overdamped Josephson junction in superconductivity. 
In \cite{bt0} they have described those complex parameters $(l,\la,\mu)$ with $\mu\neq0$ 
for which equation (\ref{heun}) has a polynomial solution: this holds exactly, when 
$l\in\nn$ and the point $(\la,\mu)$ lies on an algebraic curve $\Gamma_l\subset\cc^2_{(\la,\mu)}$ 
called the {\it spectral curve}. Namely, $\Gamma_l$ is the zero locus of the 
determinant of a remarkable 3-diagonal $l\times l$-matrix, see \cite[formula (21)]{bt0} 
and Theorem \ref{tpol} below. 

In the present paper we show that for every $l\in\nn$ the spectral curve 
$\Gamma_l$ is irreducible 
(Theorems \ref{irr}, \ref{cirr} stated in Subsection 1.1 and proved in Subsection 2.1). 

In Subsection 1.2 we state a  conjecture on formula for genera of curves $\Gamma_l$ 
and confirm it for $l\leqslant20$ via a computer-assisted proof. In Section 3 we 
discuss this  genus formula conjecture in more detail, with some more pictures for the real parts of 
the spectral curves. We prove that the conjectured genus formula provides 
an  upper estimate for the genus. We show that the genus formula conjecture 
is equivalent to  regularity of appropriate curve in $\mathbb P^1\times\mathbb P^1$ 
birationally equivalent to $\Gamma_l$.  

In Subsection 1.3 we study the real parts of the spectral curves. We prove 
that the upper half-plane $\{\mu>0\}$ intersects $\Gamma_l$ by $l$ non-intersecting 
smooth curves without vertical tangent lines. Then we deduce absence of real ovals 
in $\Gamma_l$. 

In Subsections 1.4 and 1.5 we present background material on model of Josephson junction 
and its relation to special double confluent Heun equations. This model 
is given by a family of dynamical systems on two-torus depending on three real parameters: 
one of them, the frequency $\omega>0$ is fixed; two other parameters $(B,A)$ are 
variable. The rotation number of dynamical system is a function $\rho=\rho(B,A;\omega)$. 
The phase-lock areas are those level sets $\{\rho=r\}$ 
that have non-empty interiors. They exist only for $r\in\zz$, 
see \cite{buch2}. 
It is interesting to study dependence of the rotation number on the parameters, 
the phase-lock area portrait in $\rr^2_{(B,A)}$ and its evolution, as $\omega$ changes.

{\bf Conjecturally} for every $r\in\nn$ the upper half  $L_r^+=L_r\cap\{ A\geq0\}$  of each phase-lock area $L_r$ satisfies the following statements (confirmed numerically, 
see Figures 3 and 4 taken from \cite{bg2}): 

C1) The upper area $L_r^+$  intersects the line $\La_r=\{ B=\omega r\}$ 
(which is called its axis) by a ray $Sr$ going upwards and bounded from below 
by a point $\mcp_r$; $Sr=\{\omega r\}\times[A(\mcp_r),+\infty)$ (see a partial 
result and a survey in \cite[theorem 1.12 and section 4]{g18}). 

C2) The complement 
\begin{equation} Ir:=\La_r^+\setminus Sr\,=\{\omega r\}\times[0,A(\mcp_r))\label{ir}
\end{equation}
lies on the left from the area $L_r^+$. 

C3) As $\omega>0$ decreases, in the renormalized coordinates 
$(l=\frac B\omega, \mu=\frac A{2\omega})$ the point $\mcp_r$ moves up 
and the lower part 
$L_r\cap\{0<A<A(\mcp_r)\}$ of the upper area $L_r^+$ 
  "moves from the left to the right". 

  In Subsection 1.7 we discuss a Monotonicity Conjecture on evolution of 
the family of phase-lock area portraits, as $\omega>0$ decreases. It deals with 
 the intersection points of their  boundaries  with a segment $Ir$, see (\ref{ir}), $r\in\nn$. 
Roughly speaking, the conjecture states that topologically the intersection points appear (disappear) in the same way, as if the above statement C3)  were true. More 
precisely, it states that as $\omega>0$ decreases, no new intersection point may 
be born from a part of  a boundary curve moving from the right to the left in the 
renormalized coordinates. We present a partial result towards its confirmation:  
a proof for some special points of intersection $\partial L_{s}\cap Ir$ 
with $s\equiv r(\modulo 2)$, $s\neq r$; the so called {\it generalized simple intersections} introduced in Subsection 1.5 (except for  those 
with $s=l$). Namely, for every given $l\in\nn$ and $\omega>0$ consider 
 the intersection points of the real part of the $l$-th spectral 
curve $\Gamma_l$ with the curve $\{\la+\mu^2=\frac1{4\omega^2}\}$. The generalized 
simple intersections are exactly those points in the line $\La_l$, whose $\mu=\frac A{2\omega}$-coordinates coincide with the $\mu$-coordinates of the above intersection 
points in the $l$-th spectral curve. The Monotonicity Conjecture 
will be proved in Section 4 for all the generalized simple 
intersections in  $\La_l$, except for those lying in $\La_l\cap\partial L_l$. 

In  Subsection 1.5 we state new results on generalized simple intersections (Theorem \ref{doubint} and Corollary \ref{cdoubint}).  
Corollary \ref{cdoubint} states that for every $l\in\nn$ 
and every $\omega>0$ small enough (dependently on $l$) 
there are exactly $l$ distinct generalized simple intersections 
in $\La_l$, and they depend analytically on $\omega>0$; they 
obviously tend to the $A$-axis, 
as $\omega\to0$.

The boundaries of the phase-lock areas form a countable union of analytic curves 
in $\rr^2_{(B,A)}$. Their families depending on the frequency parameter $\omega$ 
form a countable union of two-dimensional analytic surfaces in 
$\rr^3_{(B,A,\omega^{-1})}$. In Subsection 1.6 we present the following 
{\bf unexpected result} (Theorem \ref{irred2}): {\it the complexification} of 
the above {\it countable} union  of surfaces (families of boundaries)
is a two-dimensional analytic subset in $\cc^3$ consisting of just {\it four 
irreducible components!} The proof of Theorem \ref{irred2} given in 
Subsection 2.3 is based on the irreducibility of spectral curves $\Gamma_l$ 
and a result on the mentioned above simple intersections: on irreducible 
components of their complexified families (Theorem~\ref{irred1} stated in 
Subsection 1.5 and proved in Subsection 2.2).

At the end of the paper we discuss some open questions on the complexified (unions of) boundaries 
of the phase-lock areas. And also about  the complex family of double confluent 
Heun equations (\ref{heun}): on the locus  of those parameter values in $\cc^3$, 
for which the corresponding monodromy operator has a multiple eigenvalue. We provide 
a partial result on irreducible components of the latter locus, 
which is an immediate corollary of main results.

\subsection{
    Irreducibility of loci of special double confluent Heun equations 
    with polynomial solutions (A.\,A.\,Glutsyuk)
}

Let us recall the description of the parameters corresponding to equations~\eqref{heun}
with polynomial solutions.
To do this, consider the three-diagonal 
$l\times l$-matrix 

\begin{equation}
    H_l=\left(
    \begin{matrix}  0 & \mu & 0 & 0  & 0 & 0 \dots & 0\\
        \mu(l-1) & 1-l & 2\mu & 0 & 0 & \dots & 0\\
        0 & \mu(l-2) & -2(l-2) & 3\mu & 0 & \dots & 0\\
        \dots & \dots & \dots & \dots & \dots & \dots & \dots\\
        0 &\dots & 0 & 0 & 2\mu & -2(l-2) & (l-1)\mu\\
        0 & \dots & 0 & 0 & 0 & \mu & 1-l 
    \end{matrix}\right):
    \label{defh}
\end{equation}
\begin{align*}
    H_{l;ij}=0 \text{ if } |i-j|\geqslant2; & \quad
    H_{l;jj}=(1-j)(l-j+1); \\
    H_{l;j,j+1}=\mu j; & \quad
    H_{l;j,j-1}=\mu(l-j+1). \\
\end{align*}
The matrix $H_l$ belongs to the class of the so called Jacobi matrices that arise in 
different questions of mathematics and mathematical physics \cite{ilyin}. 

\begin{theorem}
    \label{tpol}\cite[section 3]{bt0}
    A special double confluent Heun equation (\ref{heun})  with 
    $\mu\neq0$  has a polynomial solution, if  and only if $l\in\nn$ and the 
    three-diagonal matrix $H_l+\la\operatorname{Id}$ has zero determinant. 
    For every $l\in\nn$ the determinant $\det(H_l+\la\operatorname{Id})$ 
    is a polynomial in $(u,v)=(\la,\mu^2)$ of degree $l$: 
    \begin{equation}
        \det(H_l+\la\operatorname{Id})=P_l(\la,\mu^2).
    \label{plmu}\end{equation}
\end{theorem}
See also \cite[remark 4.13]{bg2} for non-existence of polynomial solutions for 
$l\notin\nn$ and $\mu\neq0$. 

\begin{theorem} 
    \label{irr}
        The polynomial $P_l(u,v)$ from (\ref{plmu}) is irreducible. 
\end{theorem}

\begin{theorem}
    \label{cirr}
    For every $l$ the {\bf $l$-th spectral curve} 
    \[
        \Gamma_l:=\{ P_l(\la,\mu^2)=0\}\subset\cc^2_{(\la,\mu)}
    \]
    is irreducible.
\end{theorem}

Theorems~\ref{irr} and~\ref{cirr} will be proved in Subsection~2.1.
One can see examples of real parts of some curves~$\Gamma_l$ on Fig.\,\ref{gamma_pics}.

\begin{remark} Equation (\ref{plmu})  
defining the curve $\Gamma_l$ belongs to a remarkable 
 class of determinantal representations of plane curves: equations 
 $$\det(x_1L_1+x_2L_2+x_3L_3)=0,$$ 
 where $(x_1:x_2:x_3)\in\cp^2$ and $L_1$, $L_2$, $L_3$ are $l\times l$-matrices. Determinantal representations of curves  
 arise in algebraic geometry and integrable systems,  see 
 \cite{vin0, vin} and references therein. Complete description of 
 determinantal representations of smooth complex irreducible projective 
  curves was obtained in \cite{vin0}. Self-adjoint determinantal 
  representations of real smooth plane curves were described in 
  \cite{vin}.
\end{remark}
\begin{figure}[ht]
    \centering
    \begin{subfigure}[$\Gamma_2$]
        \centering
        \includegraphics[width=0.3\textwidth]{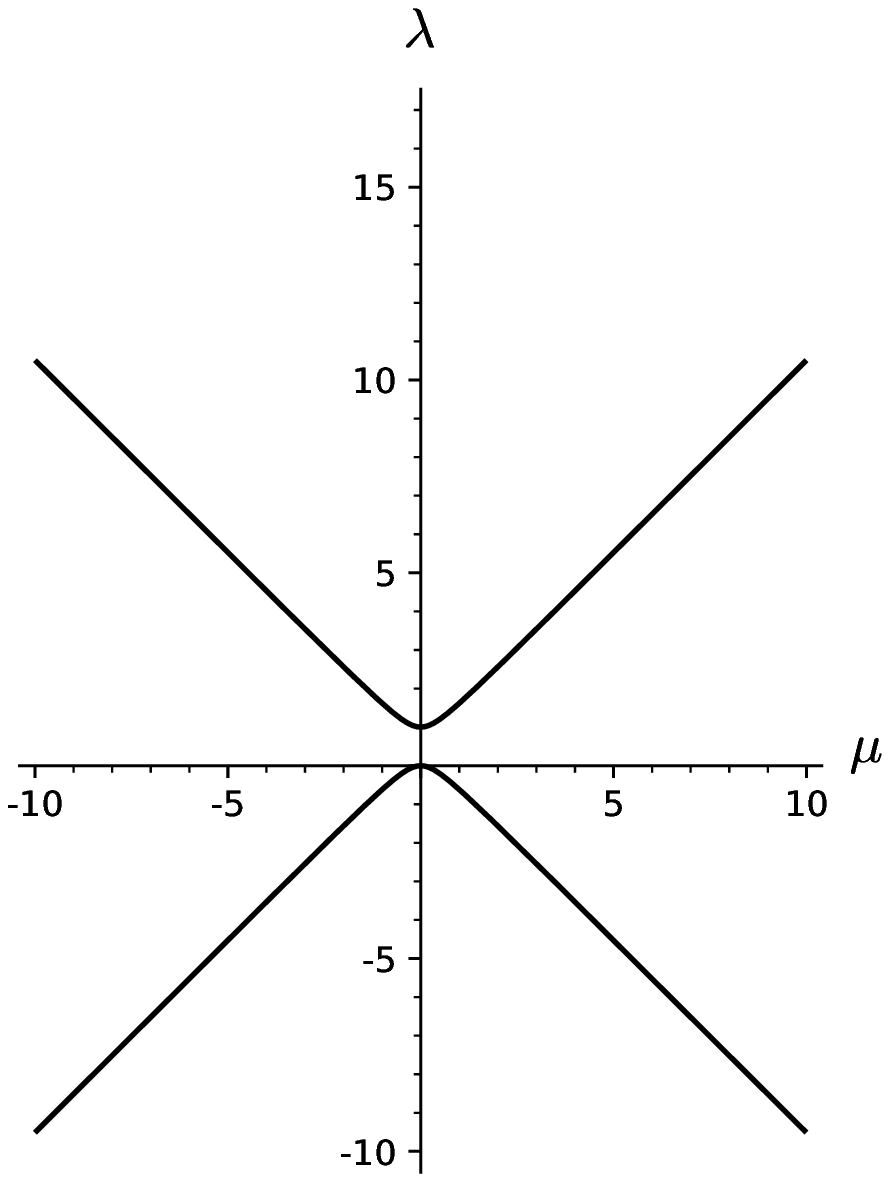}
    \end{subfigure}
    \begin{subfigure}[$\Gamma_3$]
        \centering
        \includegraphics[width=0.3\textwidth]{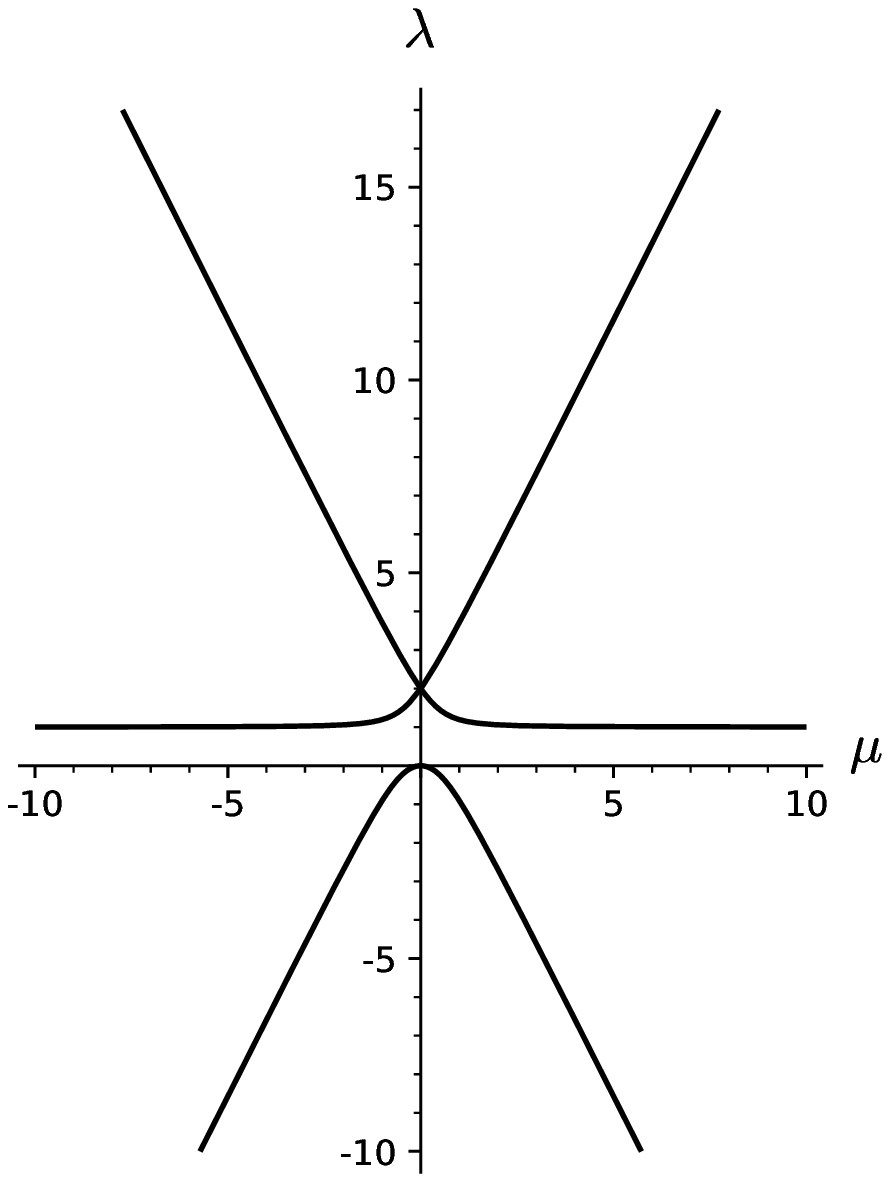}
    \end{subfigure}
    \begin{subfigure}[$\Gamma_4$]
        \centering
        \includegraphics[width=0.3\textwidth]{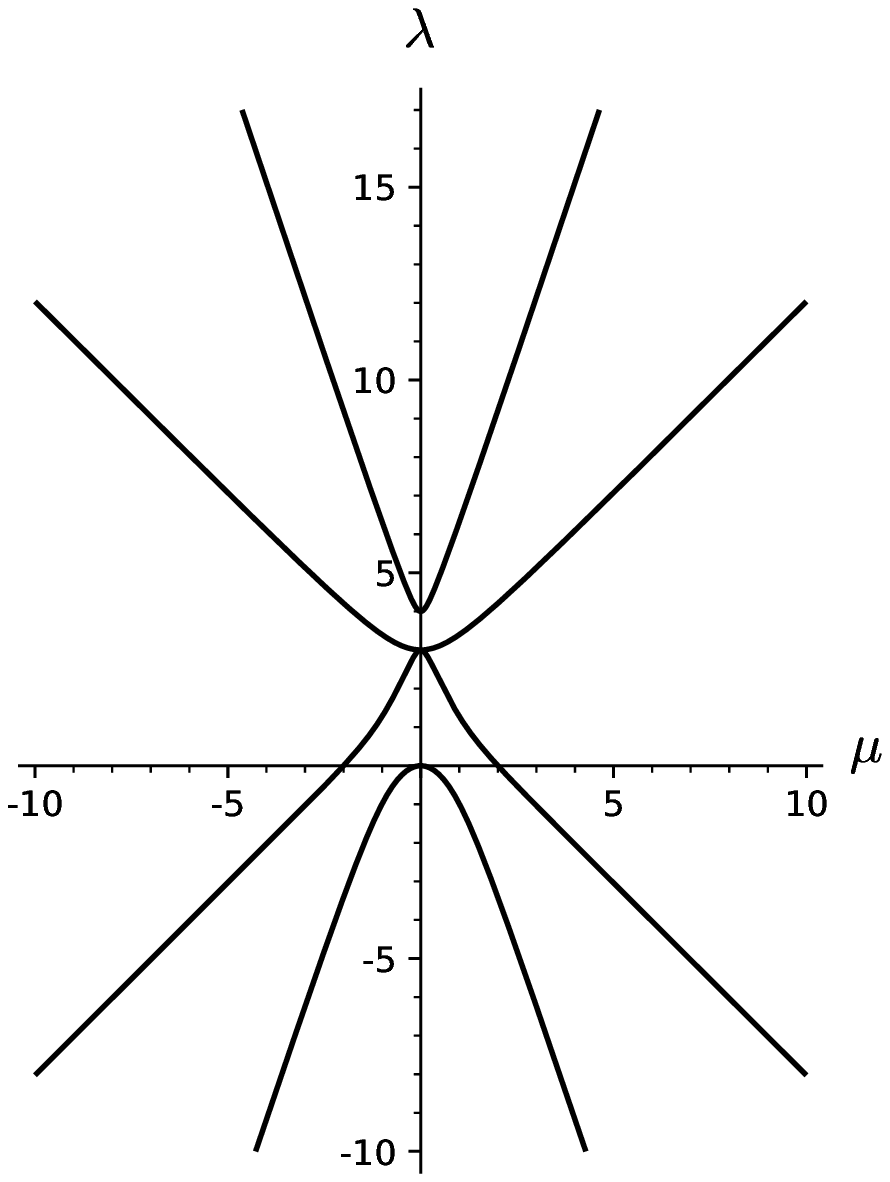}
    \end{subfigure}
    \begin{subfigure}[$\Gamma_5$]
        \centering
        \includegraphics[width=0.3\textwidth]{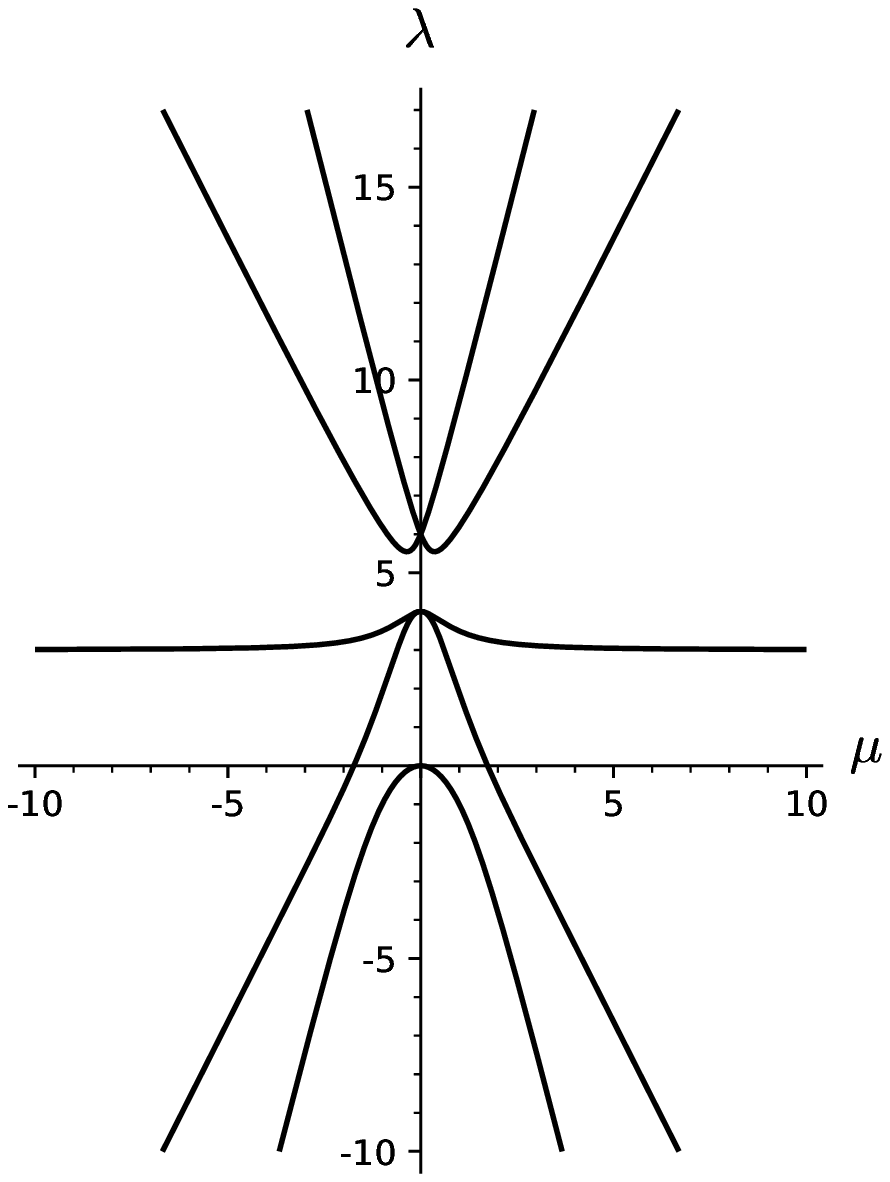}
    \end{subfigure}
    \begin{subfigure}[$\Gamma_6$]
        \centering
        \includegraphics[width=0.3\textwidth]{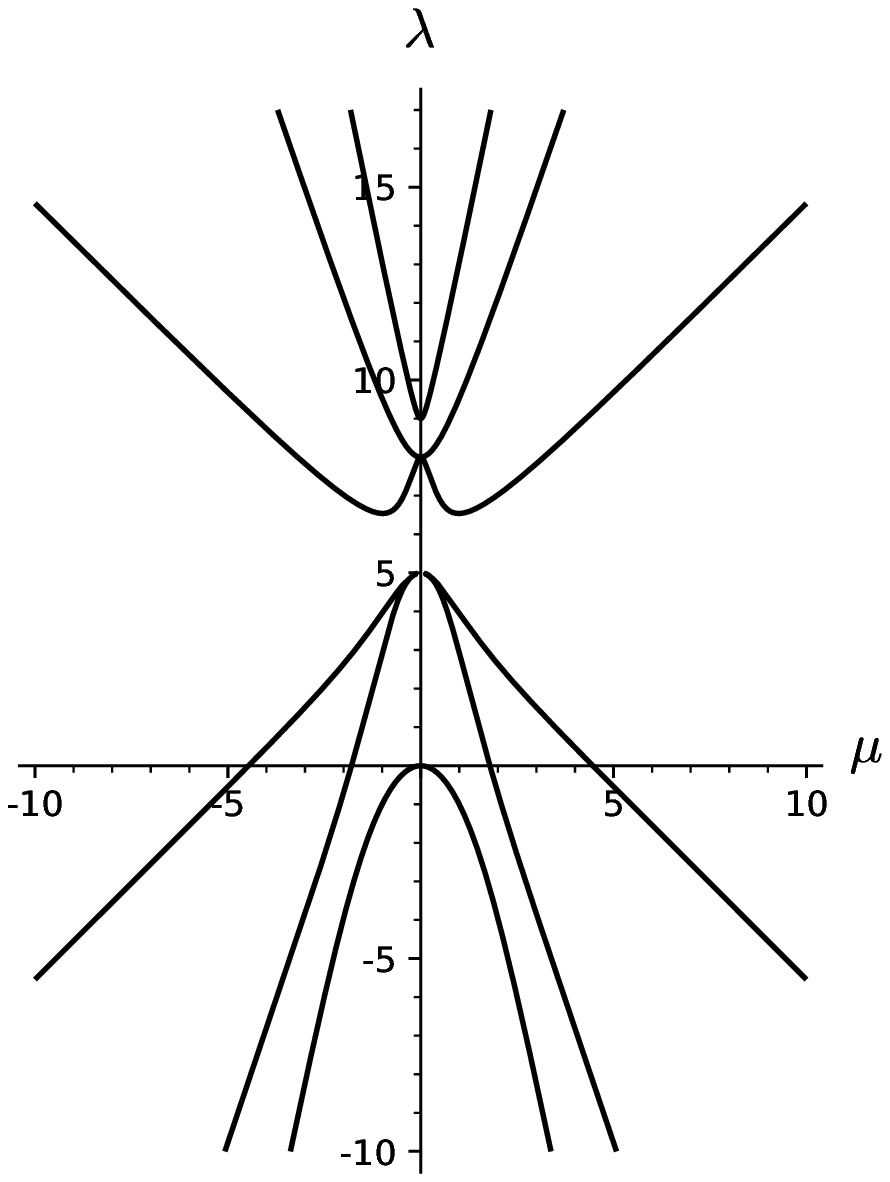}
    \end{subfigure}
    \begin{subfigure}[$\Gamma_7$]
        \centering
        \includegraphics[width=0.3\textwidth]{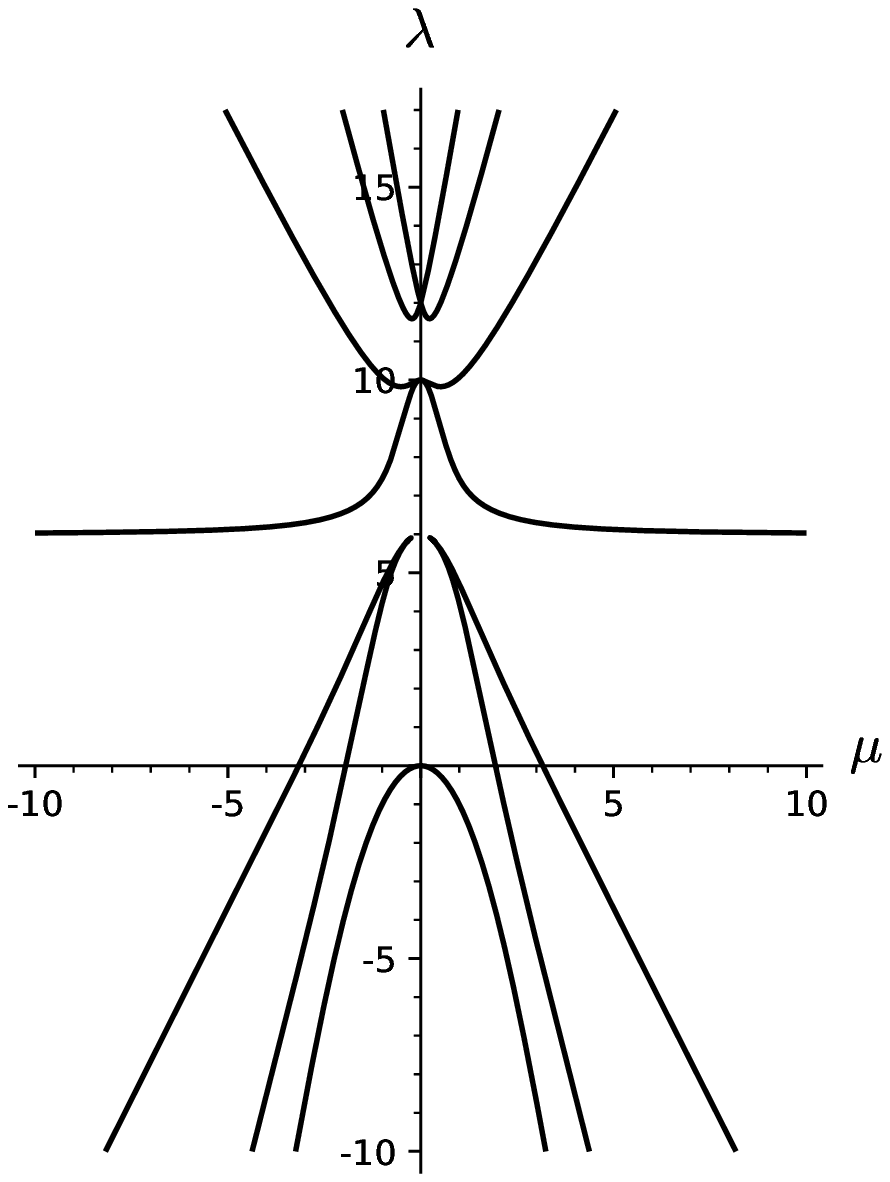}
    \end{subfigure}
    \caption{Some examples of curves~$\Gamma_l$ (I.V.Netay).}
    \label{gamma_pics}
\end{figure}

\subsection{
    Genera of spectral curves of low degrees and a genus formula conjecture (I.\,V.\,Netay)
}

Recall that the {\it geometric genus} of an irreducible algebraic curve is the 
genus of the Riemann surface parametrizing it bijectively (except for possible 
singularities), i.\,e., the genus of its normalization.
\begin{conjecture} \label{congenu}
    {\it The geometric genus of the curve $\Gamma_l$ equals}
    \[
        \begin{cases}
            \left(\dfrac{l-2}{2}\right)^2, & l\text{ even;} \\
            \\
            \left(\dfrac{l-1}{2}\right)\left(\dfrac{l-3}2\right), & l\text{ odd.} \\
        \end{cases}
    \]
\end{conjecture}

We will prove in Section 3 that the estimate from above holds, i.\,e. the genera of
curves~$\Gamma_l$ do not exceed these values (see~Proposition~\ref{GeneraMax}).

For small $l$ one can get the above formulas  directly 
(see Section 3). 
For some next values (for $l \leqslant 20$; $g(\Gamma_{20})=81$) 
the Conjecture holds and can be verified by direct
computation (the author of the present subsection 
computed this in the {\tt SageMath} open source 
computer algebra system).

\begin{figure}[ht]
\begin{center}
\begin{lstlisting}[language=Python]
K = PolynomialRing(QQ, 3, names='theta,mu,r')
theta, mu, r = K.gens()
P.<theta,mu,r> = ProjectiveSpace(QQ, 2)

d = lambda i, j: 1 if i == j else 0

def h(n,i,j):
    if i == j:
        return -i*(n - i + 1)*r
    elif i + 1 == j:
        return (i + 1)*mu
    elif i == j + 1:
        return (n - i + 1)*mu
    else:
        return 0
H = lambda n: matrix(
    [[h(n,i,j) + theta*d(i,j) 
      for j in range(n+1)] 
     for i in range(n+1)]
)

for i in range(1, 20)
    C = P.curve([K(det(H(i)))])
    print(C.geometric_genus())
\end{lstlisting}
    \caption{Listing of the code computing genera of some curves~$\Gamma_l$.}
\end{center}
\end{figure}

The function~$\tt geometric\_genus$ refers to the computation in the 
package {\tt Singular} (free available computer algebra system for polynomial 
computations with special emphasis on the needs of commutative algebra, 
algebraic geometry, and singularity theory.)
It calculates the {\it Hilbert polynomial} (see more details and examples 
in~\cite{Eisenbud}) of the normalization of the curve.

Let~$C\subset\cp^n$ be a projective curve. Then its {\it Hilbert polynomial} is
\[
    h_C(t) = d\cdot t -p_a(C) + 1,
\]
where~$d$ is the degree of the curve~$C$,~$p_a(C)$ is its arithmetic 
genus\footnote{See https://www.win.tue.nl/$\sim$aeb/2WF02/hilbert.pdf, p.5}. 
The geometric genus is the arithmetic genus of the normalization~$C_n$
of~$C$. If we are able to compute the normalization, we can compute the 
geometric genus. The most time consuming operation here is normalization.
In the case of~$\Gamma_{20}$ it takes about a few hours. See the code for computing 
genera at Fig. 2.

\subsection{
    Real spectral curve $\{ P_l(\la,\mu^2)=0\}$: topology and absence of ovals  (I.\,V.\,Netay)
}

Here we prove some results on topology of the real spectral 
curves $\Gamma_l$. 

\begin{lemma} \label{lemtop}
        The intersection of~$\Gamma_l$ with open half-plane~$\{\mu > 0\}$
        consists of~$l$ non-intersecting smooth curves without horizontal 
        tangents $\{\mu=\operatorname{const}\}$.
\end{lemma}

\begin{definition}
    An {\it oval}  of a real  planar projective algebraic curve $\Gamma$ is its 
    subset analytically parametrized by circle; the parametrization 
    should be bijective except for possible singularities of the curve. 
    (An oval may contain singularities, 
    and the parametrization may have zero derivative at some 
    singularities.) 
\end{definition}

\begin{theorem} \label{toval}
    The real curve~$\Gamma_l = \{P_l(\lambda,\mu^2) = 0\}$ has no ovals in the affine plane $\rr^2_{(\la,\mu)}$ and no ovals in 
   closure of the half-plane $\{\mu>0\}$ in the ambient 
   projective plane $\rp^2\supset\rr^2_{(\la,\mu)}$.
\end{theorem}

The above lemma and theorem are basically implied by  the following key result 
of V.\,M.\,Buchstaber and S.\,I.\,Tertychnyi and Proposition~\ref{poval} stated below.
 \begin{theorem} \label{tsimple}~\cite[p.974, theorem 1]{bt0}. 
  For $\mu\ne 0$ all the eigenvalues of the matrix $H_l$ are real and simple; that is, 
  the polynomial $P_l(\la,\mu^2)$ with fixed $\mu$ considered as a 
  polynomial in one variable $\la$ has $l$ real and simple roots.
  \end{theorem}
\begin{proof} {\bf of Lemma \ref{lemtop}.} 
        For any fixed~$\mu' > 0$ the intersection $\Gamma_l \cap \{\mu = \mu'\}$ consists $l$ of
        distinct points, by Theorem \ref{tsimple}. At the same time by  B\'{e}zout Theorem
        the intersection of a complex curve of degree~$l$ with any complex line consists of
        $l$ points if any point is counted with its multiplicity.

        For real $\mu'$ we already have~$l$ distinct  real intersection points. Also Theorem \ref{tsimple} 
        implies that all these intersections are simple. This implies that they are regular 
        points of the curve $\Gamma_l$ and at each of these points 
        the line $\{\mu=\mu'\}$ is not  tangent to $\Gamma_l$. This proves Lemma \ref{lemtop}.     
    \end{proof}

\begin{proof} {\bf of Theorem \ref{toval}.} Lemma \ref{lemtop} immediately 
implies absence of ovals in the affine plane 
$\rr^2_{(\la,\mu)}$ and in the upper half-plane $\{\mu>0\}
\subset\rr^2$. Moreover, Lemma \ref{lemtop} together with Roll  
Theorem imply that  ovals lying in its closure in $\rp^2$ (if any) 
should intersect both its boundary lines: 
the axis $\{\mu=0\}$ and the infinity line. To 
prove absence of the latter ovals, we use the following proposition.

    \begin{proposition} \label{poval}
        \label{Gamma:aymptotics}
        The asymptotic directions   of branches of the complex  curve 
        $\Gamma_{l}$ at infinity correspond to
        ratios~$\lambda/\mu$ equal $l-1, l-3,\ldots, -(l-1)$. In particular, 
        $\Gamma_l$ does not contain 
        the point of intersection of the $\lambda$-axis with the infinity 
        line in $\cp^2\supset\cc^2_{(\la,\mu)}$.  
    \end{proposition}
    \begin{proof}
        Asymptotic directions correspond to the intersection points 
        of the curve $\Gamma_{l}$ with
        the infinity line. That is, to zeros of the higher homogeneous part of the polynomial 
        $\det(H_{l}+\lambda\operatorname{Id})$, which is equal to 
         the determinant of the $l\times l$-matrix 
        \begin{equation}
            \begin{pmatrix}
                \lambda & \mu &&& \\
                (l-1)\mu & \lambda & 2\mu && \\
                & (l-2)\mu & \lambda & \ddots & \\
                && \ddots & \ddots & (l-1)\mu \\
                &&& \mu & \lambda \\
            \end{pmatrix}.
        \label{matsl}\end{equation}
        The matrix (\ref{matsl}) coincides with the matrix of the operator
        \begin{equation}
        L_{\la,\mu}:=\lambda + \mu S, \ \ S:=
      y\frac{\partial}{\partial x} +
            x\frac{\partial}{\partial y} \label{lsmu}\end{equation}
        acting on the space of homogeneous polynomials of degree~$l-1$ in~$x$ and~$y$.
        The vector fields $y\frac{\partial}{\partial x}$ and~$x\frac{\partial}{\partial y}$ 
    generate the Lie algebra $\mathfrak{sl}(2)$ 
    of divergence free linear vector fields. 
        The space of homogeneous polynomials of degree~$l-1$ as a representation
        of~$\mathfrak{sl}(2)$ is the $(l-1)$-th symmetric power of its representation
        acting on linear forms. In the space of linear forms the eigenvalues of the operator 
      $S$ in (\ref{lsmu}) are $\{-1,1\}$. Therefore the eigenvalues of its action on 
      the  $(l-1)$-th symmetric power are $l-1,l-3,\ldots,-(l-1)$. Hence, the determinant 
      of the matrix (\ref{matsl}) of the operator $L_{\la,\mu}$ is equal to 
      $(\la-(l-1)\mu)(\la-(l-3)\mu)\dots(\la+(l-1)\mu)$. This proves Proposition \ref{poval}.
    \end{proof}

The ratios $\la\slash\mu=l-1, l-3,\dots,-(l-1)$ from Proposition \ref{poval} 
corresponging to the intersection 
of the curve $\Gamma_l$ with the infinity line form $l$ distinct numbers. Therefore, 
the curve $\Gamma_l$ intersects the infinity line transversely at its 
$l$ distinct regular points  (being a curve of degree $l$). The latter 
intersection points are distinct from 
the point of intersection of the infinity line with the $\la$-axis 
(Proposition \ref{poval}). 
Hence, each local branch of the real curve $\Gamma_l$ 
at any of the latter points crosses the infinity line from the half-plane 
$\{\mu>0\}$ to the opposite half-plane $\{\mu<0\}$. This implies that $\Gamma_l$ 
cannot have ovals intersecting the infinity line and lying in the closure of the upper 
half-plane. This finishes the proof of  Theorem \ref{toval}.
\end{proof}
\begin{proposition} 
The curve~$\Gamma_l$ intersects the line $\{\mu=0\}\subset\rp^2$ at points with 
$\la$-abscissas $0$, $1\cdot(l-1)$, $2(l-2)$, \ldots, $(l-1)\cdot1$. 
All these intersection points 
are singular points of the curve $\Gamma_l$, except for the points with abscissas 
$\frac{l^2}4$ (if $l$ is even) and $0$, which are regular points of orthogonal intersection. 
 \end{proposition}
\begin{proof} The first statement of the proposition follows from definition and 
Proposition \ref{poval} (which implies that the point of the intersection of the line 
$\{\mu=0\}$ with the infinity line does not lie in $\Gamma_l$). Note that all the 
above abscissas are roots of the polynomial $P_l(\la,0)$, and the multiplicity of 
each of them equals two (except for possible roots 
$0$ and $\frac{l^2}4$, where the multiplicity is equal to 1). 
Recall that the curve $\Gamma_l$ is invariant under the symmetry 
$(\la,\mu)\mapsto(\la,-\mu)$, and hence, 
so are its germs at all the above intersection points. This implies that 
the germ  
of the curve $\Gamma_l$ at any of the points  $(\frac{l^2}4,0)$ (if $l$ is even) and 
$(0,0)$ consists of just one local branch orthogonal to the line $\{\mu=0\}$. 
The germ of the curve $\Gamma_l$ at any other intersection point 
contains at least two local branches, and hence, is singular. 
Indeed,  
otherwise if it consisted of just one regular branch, then this branch would be tangent to the line $\{\mu=0\}$ (the intersection 
point being of multiplicity two). The latter branch  
obviously cannot coincide with the line and thus, cannot 
be invariant under the above symmetry. The proposition 
is proved.
\end{proof}

\subsection{
    Model of the overdamped Josephson effect, phase-lock areas and 
    associated family of Heun equations
}

Our results are motivated by applications to the family 
\begin{equation}
    \frac{d\varphi}{dt}=-\sin \varphi + B + A \cos\omega t; \ \ \omega>0, \ B,A\in\rr, 
    \label{josbeg}
\end{equation}
of  nonlinear equations, which arises in several models in physics, mechanics and 
geometry: in a model  of the Josephson junction  in superconductivity (our main motivation), 
see~\cite{josephson, stewart, mcc, bar, schmidt, schon}, \cite{bg}-\cite{bt1}; 
in  planimeters, see~\cite{Foote, foott}. 
Here $\omega$ is a fixed constant, and $(B,A)$ are the parameters; $B$ is called 
{\it abscissa} and $A$ {\it ordinate}. Set 
\[
    \tau=\omega t, \ l=\frac B\omega, \ \mu=\frac A{2\omega}.
\]
The variable change $t\mapsto \tau$ transforms (\ref{josbeg}) to a 
non-autonomous ordinary differential equation on the two-torus 
$\mathbb T^2=S^1\times S^1=\rr^2\slash2\pi\zz^2$ with coordinates $(\varphi,\tau)$: 
\begin{equation}
    \dot \varphi=\frac{d\varphi}{d\tau}=-\frac{\sin \varphi}{\omega} + l + 2\mu \cos \tau.
    \label{jostor}
\end{equation}
The graphs of its solutions are the orbits of the vector field 
\begin{equation}
    \begin{cases} 
        & \dot\varphi=-\frac{\sin \varphi}{\omega} + l + 2\mu \cos \tau \\
        & \dot \tau=1
    \end{cases}
    \label{josvec}
\end{equation}
on $\mathbb T^2$. The {\it rotation number} of its flow, see \cite[p. 104]{arn}, is a function $\rho(B,A)=\rho(B,A;\omega)$ 
of the parameters of the vector field. It is given by the formula
$$\rho(B,A;\omega)=\lim_{k\to+\infty}\frac{\varphi(2\pi k)}{2\pi k},$$
where $\varphi(\tau)$ is an arbitrary solution of equation~\eqref{jostor}. 
(Normalization convention: the rotation number of a usual circle rotation equals the rotation angle divided by $2\pi$.)

The {\it phase-lock areas} are those level subsets of the rotation number function 
in the $(B,A)$-plane that have non-empty interiors. They have been studied 
by V.\,M.\,Buchstaber, O.\,V.\,Karpov, S.\,I.\,Tertychnyi et al, see~\cite{bg}-\cite{bt1},~\cite{4, LSh2009, IRF, krs, RK} 
and references therein. It is known that the {\it phase-lock areas exist only 
for integer values of the rotation number function \cite{buch2}}, contrarily to 
the Arnold tongues picture~\cite[p.110]{arn}.

\begin{remark}
    \label{rklim}
    Fix an $\omega>0$. For a given 
    $(B,A)\in\rr^2$ consider the time $2\pi$ flow map 
    $$h_{(B,A)}:S^1_{\varphi}\to S^1_{\varphi}$$
    of the corresponding dynamical system (\ref{josvec}) on the torus acting on the 
    transversal circle $S^1_\varphi=\{\tau=0\}=\rr\slash 2\pi\zz$. It coincides with the 
    corresponding Poincar\'e map of the transversal circle. 
    It is known, see, e.\,g., \cite{RK}, that $(B,A)$ lies in 
    the boundary of a phase-lock area, if and only if $h_{(B,A)}$  either has a parabolic 
    fixed point (and then it has a unique fixed point $\pm\frac{\pi}2(\modulo2\pi\zz$), 
    or is identity. 
\end{remark} 

In what follows the phase-lock area corresponding to a rotation number $s$ will be 
denoted by $L_s$.

It is known that the phase-lock areas satisfy the following geometric statements:

\begin{itemize}
    \item[(i)] The boundary $\partial L_s$ of the $s$-th phase-lock area is a 
        union of two graphs $\partial L_{s,\pm}$ of two analytic functions 
        $B=g_{s,\pm}(A)$, see~\cite{buch1}. This fact was later explained by 
        A.\,V.\,Klimenko via symmetry, see~\cite{RK}, where it was shown 
        that each graph $\partial L_{s,\pm}$ consists exactly of those points 
        in $L_s$ for which the corresponding~Poincar\'e map $h_{(B,A)}$ fixes 
        the point $\pm\frac{\pi}2(\modulo2\pi\zz)$. 

    \item[(ii)] The latter functions have Bessel asymptotics, as $A\to\infty$ 
    (observed and 
        proved on physical level in ~\cite{shap}, see also \cite[chapter 5]{lich},
        \cite[section\,11.1]{bar}, ~\cite{buch2006}; proved mathematically in~\cite{RK}).

    \item[(iii)] Each phase-lock area is an infinite chain of bounded domains 
        going to infinity in the vertical direction, in this chain each two 
        subsequent domains are separated by one point. This follows from Bessel 
        asymptotics, see the references mentioned in~(ii). Those  separation 
        points that lie in the horizontal $B$-axis are calculated explicitly,  
        and we call them {\it exceptional}, or {\it growth points}, see 
        \cite[corollary 3]{buch1}. The other 
        separation points  lie outside the horizontal $B$-axis and are called 
        the {\it constrictions}, see Fig. 3, 4. {\it The Poincar\'e map $h_{(B,A)}$ is identity 
        if and only if $(B,A)$ is either a constriction, or a growth point}  \cite[proposition 2.2]{4}.
        
        \item[(iv)] Each phase-lock area $L_r$ is symmetric itself with respect to the $B$-axis and is symmetric to $L_{-r}$ with respect to the $A$-axis. This follows from the fact that the transformations $(\phi,\tau)\mapsto(\phi,\tau+\pi)$ and $(\phi,\tau)\mapsto(-\phi,\tau+\pi)$ 
   inverse the sign at $\mu$ (respectively, at $l$) in the vector 
   field (\ref{josvec}), see \cite[p.76]{bg2}. 
\end{itemize}

One of the main open conjectures on the geometry of phase-lock area portrait is 
the following. 
\begin{conjecture} \label{conexp} (experimental fact, see \cite{4}). {\it In 
every phase-lock area $L_\rho$ 
all the constrictions lie in the line $\{B=\rho\omega\}$.}
\end{conjecture}
It was shown in~\cite{4} that at all the constrictions in $L_\rho$ one has 
$l=\frac B{\omega}\in\zz$, $l\equiv\rho(\modulo 2)$ and $l\in[0,\rho]$. See a survey on 
Conjecture \ref{conexp} and related open problems  in \cite{4, bg2, g18}.

\begin{figure}[ht]
    \begin{center}
        \epsfig{file=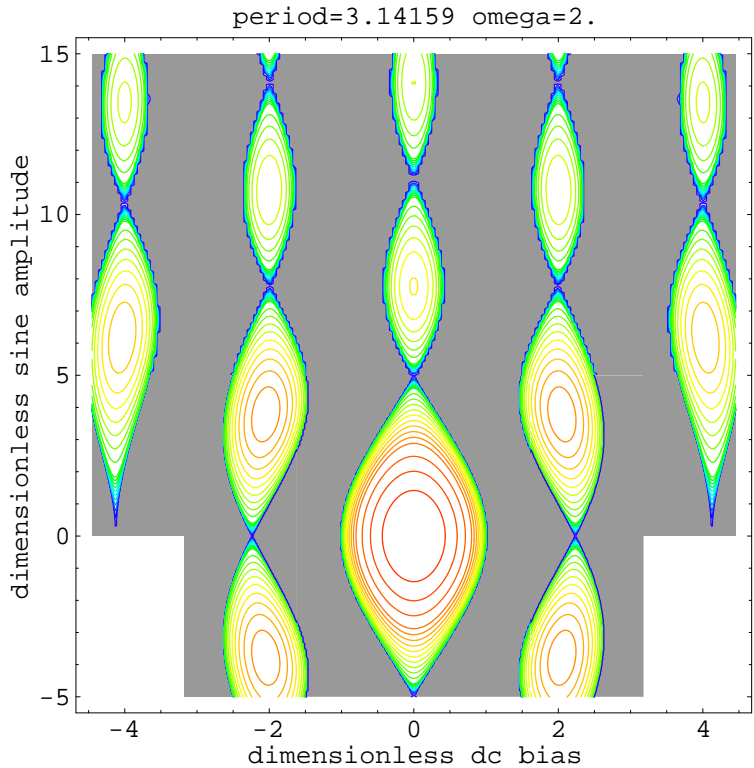}
        \caption{
            Phase-lock areas and their constrictions for $\omega=2$. 
            The abscissa is $B$, the ordinate is $A$. Figure taken from \cite[fig.\,1a)]{bg2}
        }
    \end{center}
    \label{fig:1}
\end{figure}

\begin{figure}[ht]
    \begin{center}
        \epsfig{file=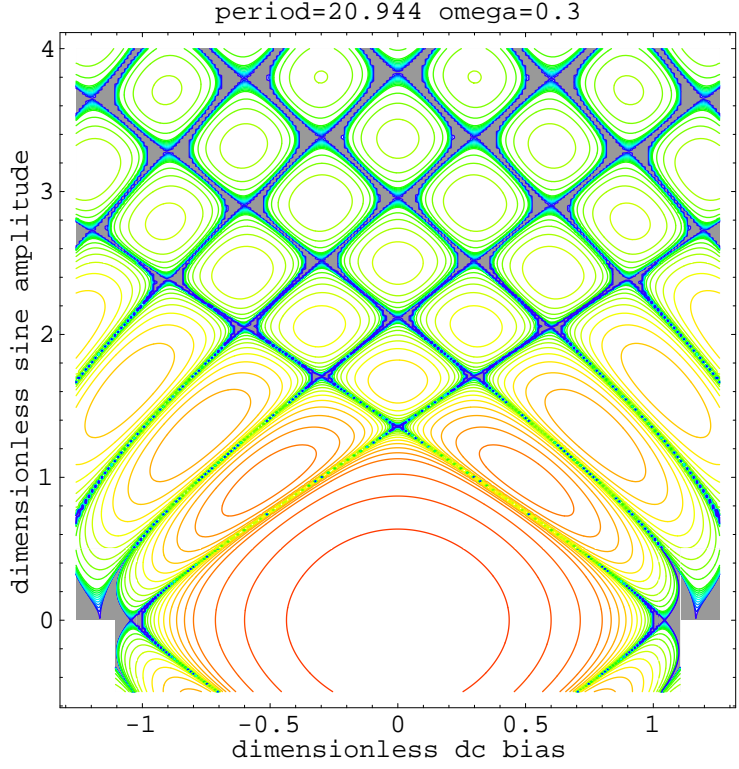}
        \caption{
            Phase-lock areas and their constrictions for $\omega=0.3$. 
            Figure taken from \cite[fig.\,1e)]{bg2}
        }
    \end{center}
    \label{fig:2}
\end{figure}

The family of non-linear equations~\eqref{josbeg} was reduced 
in~\cite{bkt1, Foote, bt1, IRF} to a family of linear equations, which were 
written in form (\ref{heun}) (with opposite sign at $l$) 
in~\cite{bt0, bt1, tert2}. Let us recall an equivalent 
reduction given in \cite[subsection 3.2]{bg}. Set 
\begin{equation}
    \Phi=e^{i\varphi}, \ z=e^{i\tau}=e^{i\omega t}.\label{phit}
\end{equation}
Considering equation~\eqref{josbeg} with complex time $t$ we get that 
transformation~\eqref{phit} sends it to the Riccati equation 
\begin{equation}
    \frac{d\Phi}{dz}=z^{-2}((lz+\mu(z^2+1))\Phi-\frac z{2i\omega}(\Phi^2-1)).
    \label{ricc}
\end{equation}
This equation is the projectivization of the following linear system in 
vector function $(u,v)$ with $\Phi=\frac v{u}$, see~\cite[subsection~3.2]{bg}:
\begin{equation}
    \begin{cases} 
        & v'=\frac1{2i\omega z}u\\
        & u'=z^{-2}(-(lz+\mu(1+z^2))u+\frac z{2i\omega}v)
    \end{cases}
    \label{tty}
\end{equation}
Set 
\begin{equation}
    \lambda=\left(\frac1{2\omega}\right)^2-\mu^2.
    \label{param}
\end{equation}
It is easy to check that a function~$v(z)$ is the component of a solution of 
system~\eqref{tty}, if and only if the function $E(z)=e^{\mu z}v(z)$ 
satisfies the second order linear equation obtained from (\ref{heun}) by changing 
sign at $l$: 
\begin{equation} 
    \mcl E=z^2E''+((l+1)z+\mu(1-z^2))E'+(\la-\mu(l+1)z)E=0.
    \label{heun2*}
\end{equation}

\begin{remark}
    Heun equations~\eqref{heun} and~\eqref{heun2*} corresponding to  the 
    family~\eqref{josvec} of dynamical systems on torus are those corresponding 
    to real parameters $l$, $\omega$, $\mu$, and thus, real $\la$. In the 
    present paper we are dealing with general  equation~\eqref{heun}, with 
    arbitrary complex parameters~$\la$, $\mu$.
\end{remark}

\def\mon{\operatorname{Mon}}

\subsection{Generalized simple intersections and polynomial solutions (A.\,A.\,Glutsyuk)}

Let us recall the following definition.

\begin{definition}\cite[definition 1.3]{g18}
    The {\it axis} of a phase-lock area $L_s=L_s(\omega)$ is the vertical line 
    \[
        \La_s=\La_s(\omega)=\{ B=s\omega\}\subset\rr^2_{(B,A)}; \ \ s\in\zz.
    \]
\end{definition}

\begin{definition} (compare with \cite[definition 1.9]{g18}) 
    A {\it generalized simple intersection} is a point of intersection of 
    an axis $\La_l$ with the boundary of a phase-lock area $L_s$ with 
    $s\equiv l(\modulo 2)$ that is not a constriction, see Fig.\,\ref{fig:3}.
\end{definition}
\begin{theorem} \label{bt02} \cite{bt0, bg2} 
    The generalized simple intersections with $B\geqslant0$ and $A\neq0$ correspond exactly 
    to those real parameter values $(B,A; \omega)$ with $\omega,B>0$, $A\neq0$  
    for which  Heun equation~\eqref{heun} with $l=\frac B\omega$, 
    $\mu=\frac A{2\omega}$, $\la=\frac1{4\omega^2}-\mu^2$ 
    has a polynomial solution. Generalized simple intersections 
    in $\La_l$ may exist only for $l\in\zz\setminus\{0\}$ 
    and those $s\equiv l(\modulo 2)$ that lie in the segment $[0,l]$. 
\end{theorem}

\begin{figure}[ht]
    \begin{center}
    \epsfig{file=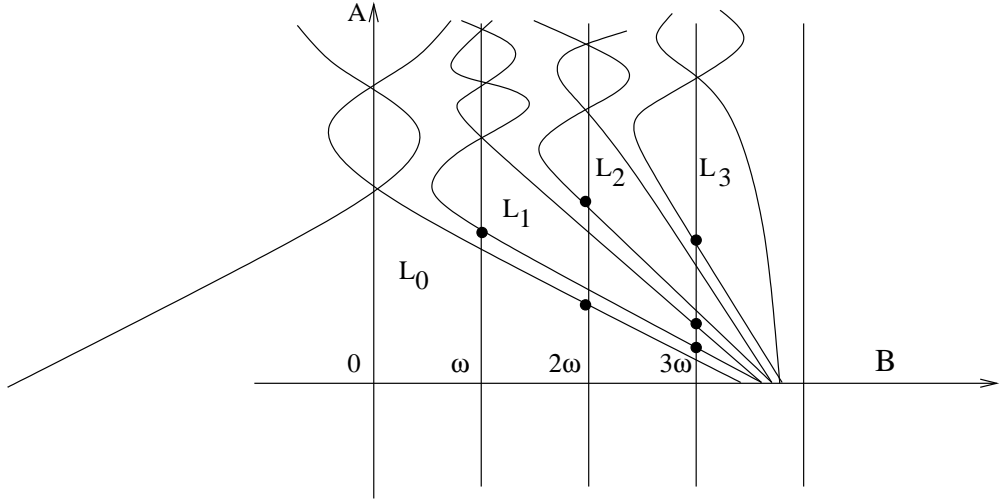}
    \caption{Approximate phase-lock areas for $\omega\simeq 0.27$; 
        the marked points are generalized simple intersections. 
        This hand-made figure, which is taken from~\cite[fig.\,2]{bg2}, 
        illustrates open Conjecture \ref{conexp}: for every $r\in\zz$ 
       all  the constrictions of the phase-lock area $L_r$ lie in 
       its axis $\La_r$. }
    \label{fig:3}
    \end{center}
\end{figure} 
 
 We prove the following new result on generalized simple intersections lying in {\it upper semiaxis}
$$\La_l^+=\La^+_l(\omega):=\La_l(\omega)\cap\{ A\geq0\}.$$
 
 \begin{theorem}
    \label{doubint}
    For every $l\in\nn$ and every $\omega>0$ the semiaxis 
    $\La_l^+=\La_l^+(\omega)$ 
    contains at least one generalized simple intersection lying in 
    $\partial L_l=\partial L_l(\omega)$. 
    For every $l,s\in\nn$, $l\geqslant3$, $s\equiv l(\modulo 2)$, $0<s<l$, and every $\omega>0$ 
    small enough (dependently on $l$) the semiaxis $\Lambda_l^+$ intersects each 
    component $\partial L_{s,\pm}$ of the boundary of the phase-lock area $L_s$. 
    For every  $l\geqslant2$ and every $\omega>0$ small enough the boundary 
    component $\partial L_{0,+}$ intersects $\Lambda_l^+$. 
    See Fig.\,5.
\end{theorem}

\begin{corollary} \label{cdoubint} For every $l\in\nn$ and every $\omega>0$ small enough (dependently on 
$l$) the semiaxis $\La_l^+(\omega)$ contains exactly $l$ distinct generalized simple intersections; exactly one of them lies in 
$\partial L_l$. 
They depend analytically on $\omega$ and tend to the $A$-axis, 
as $\omega\to0$: their common abscissa $\omega l$ tends to zero.
\end{corollary}
 
Set
\[
    r=\frac1{2\omega}; \ \ \mu=\frac A{2\omega}=Ar.
\]
For given $r>0$ (or equivalently, $\omega>0$) and $l\in\nn$, the set of the 
$\mu$-coordinates of the generalized simple intersections in the axis $\La_l(\omega)$ will 
be denoted by $\Simple_l(r)$. We set
\[
    \si_l:=\cup_{r>0}(\Simple_l(r)\times\{r\})\subset\rr^2_{(\mu,r)}.
\]

Recall that for every $(\mu,r)\in\si_l$ the corresponding point 
$(B=\frac l{2r}, A=\frac\mu r)$ is not a constriction, but a priori it can be a growth point. 
 The complement of the set $\si_l$ to growth points is an open and 
dense subset $\si_l^0\subset\si_l$ consisting of those 
parameters for which the  corresponding Poincar\'e map $h_{(B,A)}$ has a unique fixed point $\pm\frac{\pi}2(\modulo2\pi\zz)$, see 
Remark~\ref{rklim} and statement (iii) in Subsection 1.4. 
Thus, $\si_l^{0}$   is split into two disjoint parts $\si^0_{l,\pm}$  of those 
parameter values for which the above fixed point is $\pm\frac{\pi}2(\modulo2\pi\zz)$. 
Set 
\[
    \hsi_{l}:=\text{ the complex Zariski closure of the subset } \si_l 
    \text{ in } \cc^2_{(\mu,r)},
\]
\begin{equation}
    \hsi_{l,\pm}:=\text{ the complex Zariski closure of the subset } \si_{l,\pm}^0.
\label{sipm}\end{equation}

\begin{theorem}
    \label{irred1}
    For every $l\in\nn$ the image of the subset $\hsi_{l}\subset\cc^2_{(\mu,r)}$ 
    under the mapping $(\mu,r)\mapsto (\mu, r^2)$ is an irreducible 
    algebraic curve. The algebraic subset $\hsi_{l}$  consists of two 
    irreducible components: the algebraic curves $\hsi_{l,\pm}$.
\end{theorem}
Theorems \ref{doubint}, \ref{irred1} and Corollary \ref{cdoubint} will be proved in Subsection 2.2, together with  the following new result 
used in the proof of Theorem \ref{irred1}. 

\begin{theorem}
        \label{intch}
        For every $s\in\zz$ the symmetry $(B,A)\mapsto(B,-A)$ 
        with respect to the $B$-axis preserves the boundary curves 
        $\partial L_{s,\pm}$ for even $s$ and interchanges them for odd $s$. In particular, for every even $s$ the curves $\partial L_{s,\pm}$ 
        are orthogonal to the $B$-axis at the growth point of the 
        phase-lock area $L_s$. The symmetry $(B,A)\mapsto(-B,A)$ 
        with respect to the $A$-axis 
        permutes the curves $L_{s,\pm}$ and $L_{-s,\mp}$ for even  
        $s$ and the curves $L_{s,\pm}$, $L_{-s,\pm}$ for 
        odd $s$. 
    \end{theorem}

\subsection{
    Boundaries of phase-lock areas and the complexified union of their families
    (A.\,A.\,Glutsyuk)
}

For every $\omega>0$ let 
$\mcl_{\pm}^{odd}(\omega), \mcl_{\pm}^{even}(\omega)\subset\rr^2_{(B,A)}$ 
denote the unions of those points in the boundaries of the phase-lock areas 
with odd (respectively, even) rotation numbers, for which the corresponding 
Poincar\'e map $h_{(B,A)}$ fixes the point $\pm\frac{\pi}2(\modulo2\pi\zz)$. 
Recall that we denote $r=\frac1{2\omega}$. Set 
\[
    \mcl_{\pm}^{odd (even)}:=\bigcup_{\omega>0}
    (\mcl_{\pm}^{odd (even)}(\omega)\times\{ r\})\subset\rr^3_{(B,A,r)},
\]
\[
    \hmcl_{\pm}^{odd (even)}:=\text{ the minimal  analytic subset in } \cc^3_{(B,A,r)} 
    \text{ containing }\mcl_{\pm}^{odd (even)}.
\]

\begin{theorem}
    \label{irred2}
    The analytic subsets $\hmcl_{\pm}^{odd}$, $\hmcl_{\pm}^{even}$ are distinct 
    two-dimensional irreducible analytic subsets. The union of these four 
    irreducible analytic subsets is the minimal complex analytic subset 
    in $\cc^3_{(B,A,r)}$ containing $\bigcup_{\omega>0}\bigcup_{s\in\zz}(\partial L_s(\omega)
    \times\{ \frac1{2\omega}\})$. 
\end{theorem}

Theorem~\ref{irred2} will be proved in Subsection~2.3.

\subsection{ On family of phase-lock area portraits, the Monotonicity Conjecture and 
constrictions (A.A.Glutsyuk)}

In this subsection we consider the model of Josephson effect and 
study the evolution of the phase-lock area portrait, as $\omega>0$ decreases. 
We present  the Monotonicity Conjecture, its 
relation with Conjecture \ref{conexp} on abscissas of constrictions and a 
partial result towards  confirmation of the Monotonicity Conjecture. 

\begin{definition} Let $\omega_0>0$, $l\in\nn$, $\mu>0$, $B_0=l\omega_0$, 
$A_0=2\mu\omega_0$. Let $s\in\zz_{\geqslant0}$. Let 
$(B_0,A_0)\in\partial L_{s,\pm}(\omega_0)$ for 
appropriately chosen sign $\pm$. Let $\partial L_{s,\pm}(\omega_0)$ be tangent to 
the axis $\La_l(\omega_0)=\{ B=l\omega_0\}$ at $(B_0,A_0)$. Let the germ of the curve 
$\partial L_{s,\pm}$ at $(B_0,A_0)$  lie on the right from the line $\La_l(\omega_0)$. 
We say that the point $(B_0,A_0)\in\rr_+^2$ is a {\it left-moving tangency} 
(for $\omega=\omega_0$), if 
the point $(B_0,A_0)$ of  intersection $\partial L_{s,\pm}(\omega)\cap\La_l(\omega)$ 
with $\omega=\omega_0$ 
disappears for $\omega>\omega_0$: there exists a neighborhood $U$ of the 
point $(B_0,A_0)$ such that for every $\omega>\omega_0$ close enough to $\omega_0$ 
(dependently on $U$) the intersection $\partial L_{s,\pm}(\omega)\cap\La_l(\omega)$ 
contains no points in $U$.
\end{definition}

\begin{conjecture} \label{moncon} (Monotonicity Conjecture). {\it There are 
no left-moving tangencies in model of Josephson effect.} 
\end{conjecture}

\begin{remark} The Monotonicity Conjecture means that as $\omega>0$ 
decreases, the intersections of 
the axes $\La_k(\omega)$, $k\in\nn$, with boundaries of phase-lock areas 
cannot be born  from those pieces of  boundaries that move "from the right to the left" in 
the renormalized coordinates $(l=\frac B\omega,\mu=\frac A{2\omega})$. That is why 
we call it the Monotonicity Conjecture. 
\end{remark}
\begin{remark} \label{simpoc} Let $s\equiv l(\modulo 2)$, $s\geqslant0$, $\omega>0$. Every point of intersection 
 $\La_l(\omega)\cap\partial L_{s,\pm}(\omega)$ (if any; in particular, every 
 left-moving tangency lying there)  is either 
a generalized simple intersection, or a constriction, 
see Figures 6 and 7. This follows from definition.  If $s\equiv l+1(\modulo 2)$, then the 
above intersection points are neither generalized simple 
intersections (by definition), nor  constrictions (by  \cite[theorem 3.17]{4}). 
\end{remark}

\begin{proposition} Let $l\in\nn$, $s\equiv l(\modulo 2)$, $\omega_0>0$, 
and let some left-moving tangency 
$(B_0,A_0)$ for $\omega=\omega_0$ be a constriction. 
Then for every $\omega<\omega_0$ 
close enough to $\omega_0$ the points of intersection $\La_l(\omega)\cap
\partial L_{s,\pm}(\omega)$ newly born from $(B_0,A_0)$ (if any) are constrictions. 
See Fig. 7.  
\end{proposition}
\begin{proof} If for every $\omega<\omega_0$ close 
enough to $\omega_0$ some of the above newly born intersection points were 
not a constriction, then it would be a generalized simple intersection, see 
Remark \ref{simpoc}. Hence, the corresponding double confluent Heun 
equation (\ref{heun}) would have a polynomial solution. Passing to limit, as 
$\omega\to\omega_0$, 
we get that equation (\ref{heun}) corresponding to a constriction $(B_0,A_0)$ also 
has a polynomial solution. On the other hand, the corresponding equation (\ref{heun2*}) 
has an entire solution, since $(B_0,A_0)$ is a constriction, see \cite[p.\,332,~\S2]{bt1}, \cite[theorem~3.3 and subsection~3.2]{bg}. But the above properties of equations 
(\ref{heun}) and (\ref{heun2*}) to have polynomial (entire) solution are incompatible 
for real parameter values, by \cite[theorem 3.10]{bg}. The contradiction thus obtained 
 proves the proposition.
 \end{proof}
 
 \begin{figure}[ht]
    \begin{center}
        \epsfig{file=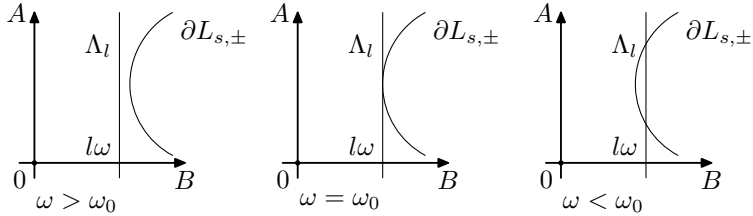}
        \caption{Left-moving tangencies (conjecturally non-existing) that are simple intersections}
    \end{center}
    \label{fig:left1}
\end{figure} 
 \begin{figure}[ht]
    \begin{center}
        \epsfig{file=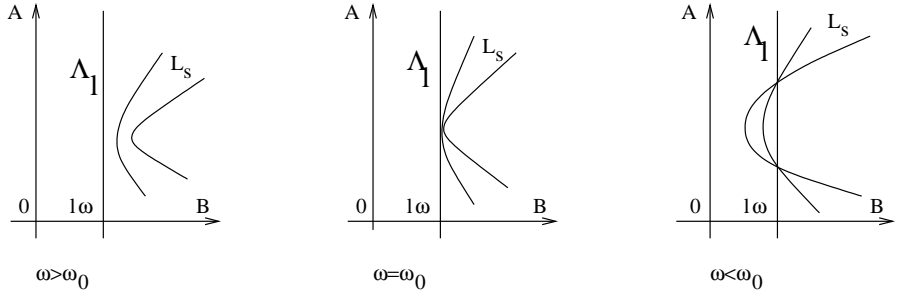}
        \caption{Left-moving tangencies (conjecturally non-existing) that are 
        constrictions}
    \end{center}
    \label{fig:left2}
\end{figure} 

 \begin{proposition} \cite[proposition 4.7]{g18} Let no left-moving tangency (if any) be a constriction. Then 
 Conjecture \ref{conexp} holds.
 \end{proposition}
 
 \begin{corollary} Monotonicity Conjecture \ref{moncon} 
 implies Conjecture \ref{conexp}.
 \end{corollary}
 
 The main result of the present subsection is the following theorem giving a 
 partial positive result towards the Monotonicity Conjecture. 
 
 \begin{theorem} \label{tmon} No left-moving tangency of a boundary curve 
 $\partial L_{s,\pm}$ with $\La_l$, $s\neq l$, can be a generalized simple 
 intersection.
 \end{theorem}
 
 Theorem \ref{tmon} is proved in Section 4.

\subsection{Historical remarks}

Model of overdamped Josephson junction and its phase-lock area portraits were studied 
in \cite{bg} -- \cite{bt1}, \cite{4, g18, LSh2009, IRF, krs, RK, tert2}. 
In the case, when $\omega>0$, $B\geqslant0$ and $A\neq0$, Buchstaber and Tertychnyi have  shown that the 
constrictions in model of Josephson junction correspond exactly to those parameter values $(B,A)$ with $A\neq0$, for which $l=\frac B\omega$ is 
integer and  equation~\eqref{heun2*}  has a non-trivial holomorphic solution 
at 0 (which is automatically an entire solution: holomorphic on $\cc$); see the 
statement in \cite[p.\,332,~\S2]{bt1} and the proof in~\cite[theorem~3.3 and subsection~3.2]{bg}. 
They have explicitly constructed a family of holomorphic solutions for 
parameters satisfying an explicit functional equation $\xi_l(\la,\mu)=0$, 
see~\cite[theorem~2]{bt1}. They have conjectured that the latter functional 
equation describes the constrictions completely. They have reduced this 
conjecture to another one saying that if equation~\eqref{heun} has a polynomial 
solution, then equation~\eqref{heun2*} does not have an entire solution. 
Both conjectures were proved in~\cite{bg}.

V.\,M.\,Buchstaber and S.\,I.\,Tertychnyi have constructed symmetries of 
double confluent Heun equation~\eqref{heun}~\cite{bt1, bt3}. 
The symmetry $\sharp:E(z)\mapsto 2\omega z^{-l-1}(E'(z^{-1})-\mu E(z^{-1}))$, which 
is an involution of its solution space, was constructed 
in \cite[equations~(32),~(34)]{tert2}. It corresponds to the symmetry 
$(\varphi,t)\mapsto(\pi -\varphi,-t)$ of the nonlinear equation (\ref{josbeg}); 
the latter symmetry was found in \cite{RK}. In \cite{bt3} they have found new 
nontrivial symmetries in the case, when $l\in\nn$ and 
equation (\ref{heun}) has no polynomial solutions. 

For a survey of Conjecture \ref{conexp} on abscissas of constrictions 
and related results see \cite{bg2, 4, g18} and references therein. 

Numerical experiences made by V.M.Buchstaber, 
S.I.Tertychnyi, I.V.Schurov, V.A.Kleptsyn, D.A.Filimonov allowed  to  observe  that as $\omega\to0$,  the "upper part" of the 
phase-lock area portrait converges to a parquet-like structure 
in the renormalized coordinates $(l,\mu)$.  
More precisely,  the complement to the union of the phase-lock 
areas is a very thin subset whose upper part tends to the 
boundary of  parquet pieces.  See Fig. 4 and also \cite{krs}. This 
statement is an open problem. 
In \cite{krs} V.A.Kleptsyn, O.L.Romaskevich and I.V.Schurov 
proved results  on 
smallness of gaps between the phase-lock areas  for small $\omega$ and the rate of their convergence to zero, as $\omega\to0$, using methods of  slow-fast systems. 

\section{Complex spectral curves and boundaries. Proofs of Theorems \ref{irr}, \ref{cirr}, \ref{doubint}, 
    \ref{irred1}, \ref{irred2} (A.\,A.\,Glutsyuk)
}

\subsection{
    Irreducibility of determinant. Proof of Theorems~\ref{irr} and~\ref{cirr}
}

First let us prove Theorem \ref{irr}. For $l=1$ the irreducibility is obvious, 
since the determinant under question coincides with the monomial $\la$. Thus, 
everywhere below we consider that $l\geqslant2$. We  prove irreducibility of the 
determinant, by finding its Newton diagram at $0$ in appropriate new affine 
coordinates $(\la,R)$. 

Conjugation  by diagonal matrix $\diag(1,\mu^{-1},\dots,\mu^{1-l})$ 
transforms the matrix $H_l$ to the same matrix without multipliers $\mu$ below 
the diagonal and with the multipliers $\mu$ above the diagonal being replaced 
by $\mu^2$. The above conjugation together with subsequent substitution 
\[
    \mu^2=R-\la, \ R:=r^2=\left(\frac1{2\omega}\right)^2
\]
transform the matrix $H_l+\la Id$ to the matrix 
\begin{equation}
    M(\la,R)=M_1(\lambda)+M_2(R), 
    \label{matm}
\end{equation}
\[
    M_1(\la) := 
    \left(
        \begin{matrix}
            \la & -\la & 0 & 0 &\dots & 0\\
            l-1 & \la-(l-1) & -2\la & 0 & \dots & 0\\
            \dots & \dots & \dots & \dots & \dots\\
            0 &\dots & 2 & \la-2(l-2) & -(l-1)\la\\
            0 & \dots & 0 & 1 & \la-(l-1)
        \end{matrix}
    \right),
\]
\[
    M_2(R)=
    \left(
        \begin{matrix}
            0 & R & 0  & 0 & \dots & 0\\
            0 & 0 & 2R & 0 & \dots & 0\\
            \dots & \dots & \dots & \dots &  \dots & (l-1)R\\
            0 & \dots & \dots &\dots &\dots & 0
        \end{matrix}
    \right).
\]

\begin{proposition}
    \label{Newton1n}
    The determinant $\det M(\la,R)$ is a polynomial containing the 
    monomials $\la^l$ and $R$ with non-zero coefficients and containing no 
    monomials $\la^k$ with $k\neq l$. 
\end{proposition}

\begin{proof}
    The coefficients at monomials $\la^k$ in the determinant $\det M(\la,R)$ 
    are equal to the corresponding coefficients in the determinant $\det M_1(\la)$. Let us 
    calculate $\det M_1(\la)$. Let us replace the second column of the matrix $M_1(\la)$ by 
    its sum with the first column. This cancels the first element $-\la$ and the free term 
    $-(l-1)$ in the second element: the rest of the second column remains unchanged. 
    In the new matrix thus obtained let us replace the third column by its sum with 
    the second column multiplied by two. This cancels its second element $-2\la$ 
    and the free term $-2(l-2)$ in the third element. Repeating similar operations, we 
    finally get a lower-triangular matrix with diagonal elements $\la$. We finally get that 
    \begin{equation}
        \det M_1(\la)=\la^l.
        \label{detm1}
    \end{equation}
    This implies that the only monomial $\la^k$ entering the polynomial $\det M(\la,R)$ 
    with non-zero coefficient is $\la^l$. 

    Now it remains to show that the monomial $R$ enters the polynomial $\det M(\la,R)$ 
    with non-zero coefficient. To do this, consider the auxiliary 
    matrix 
    \[
        N(R)=
        \left(
            \begin{matrix}
                0 & R & 0 & 0 & \dots & 0\\
                l-1 & 1-l & 2R & 0 & \dots & 0\\
                \dots & \dots & \dots & \dots & \dots & \dots\\
                0 & \dots & 0 & 2 & -2(l-2) & (l-1)R\\
                0  & \dots & 0 & 0 &  1 & 1-l
            \end{matrix}
        \right),
    \]
    whose above-diagonal elements (elements next to the diagonal)
    are  equal to $jR$, $j=1,\dots,l-1$, and whose other elements are equal to the corresponding 
    free terms in the entries of the matrix $M_1(\la)$. The coefficient at the 
    monomial $R$ is the same in both polynomials $\det M(\la,R)$ and $\det N(R)$. Thus, 
    it suffices to show that the corresponding coefficient in $\det N(R)$ is non-zero. 
    The only non-zero element 
    in the first line of the matrix $N(R)$ is its second element $R$. 
    Therefore, its determinant is equal to the product of the number 
    $-(l-1)R$ and its principal minor formed by its columns number $3,4,\dots,l$. 
    The non-zero free terms in the latter minor exist exactly in the diagonal 
    and immediately below it. This implies that the coefficient at the monomial 
    $R$ in the polynomial $\det N(R)$ is equal to the product of the number 
    $1-l$ and the diagonal terms in the above minor. The latter product is 
    non-zero. Therefore, the coefficient at $R$ in the determinant $\det N(R)$ 
    (and hence, in $\det M(\la,R)$) is non-zero. The proposition is proved.
\end{proof}

\begin{corollary} \label{cnewton}
    The Newton diagram of the polynomial $Q(\la,R)=\det M(\la,R)$ 
    consists of just one edge $E$ with vertices $(l,0)$ and $(0,1)$. 
\end{corollary}

The corollary follows immediately from the proposition and the definition of 
Newton diagram.

\begin{proof} {\bf of Theorem~\ref{irr}.}
    Suppose the contrary: the polynomial $P_l(u,v)$ is not irreducible. Then the 
    above polynomial $Q(\la,R)$, which is obtained from the polynomial $P(u,v)$ 
    by affine variable change $(u,v)=(\la,R-\la)$, is also not irreducible: it is 
    a product of two 
    polynomial factors $P_1$ and $P_2$. At least one of them, say $P_1$, 
    vanishes at 0. The Newton diagram of each polynomial factor vanishing at 
    $0$ should consist of an edge parallel to the above edge $E$ and have 
    vertices with integer non-negative coordinates. Moreover, the latter edge 
    should either coincide with $E$, or lie below $E$. This follows from the 
    well-known fact that the upper component of the Newton diagram of product 
    of two germs of analytic functions is the Minkovski sum of analogous 
    components of the factors. But the only possible edge parallel to $E$ with 
    integer vertices and lying no higher than $E$ in the positive quadrant is the edge 
    $E$ itself, since one of its vertices is $(0,1)$. 

    Therefore, only $P_1$ vanishes at $0$, and  its Newton diagram coincides 
    with the edge $E$. Thus, $P_1$  contains the monomials $\la^l$ and $R$. The 
    polynomial $P_2$ cannot be non-constant: otherwise the product $Q=P_1P_2$ 
    would have degree greater than $l$ in $(\la, R)$, while $\det M(\la,R)=Q(\la,R)$ 
    is clearly a polynomial of degree $l$. This proves irreducibility of the 
    polynomial $Q(\la,R)$, and hence, $P(u,v)$. Theorem \ref{irr} is proved.
\end{proof}

\begin{proof} {\bf of Theorem~\ref{cirr}.} The case, when $l=1$, being obvious ($P_1(\la,v)=\la$), we consider that $l\geq2$. 
    The curve $\Gamma_l$ is the preimage of the irreducible zero locus 
    $W_l=\{ P_l(\la,v)=0\}$ under the degree two mapping 
    $F\colon(\la,\mu)\mapsto(\la,\mu^2)$. Therefore, $\Gamma_l$ is an algebraic 
    curve consisting of either one, or two irreducible components. To show that 
    it is just one component, it suffices to show that the two preimages of 
    some point in $W_l$ close to the origin are connected by a path lying in 
    the regular part of the curve $W_l$. The germ at 0 of the curve $W_l$ is 
    regular and tangent to the line $R=\la+v=\operatorname{const}$: the  linear part at the 
    origin of the polynomial $P_l$ is equal to $R$ times a non-zero constant, by Corollary \ref{cnewton} and since $l\geq2$. 
    Thus, the germ under question is transversal to the line $\{ v=0\}$: the 
    critical value line of the mapping $F$. Therefore, a small circuit around 
    the origin in the above germ lifts via $F$ to a path  in the regular part 
    of the curve $W_l$ that connects two different preimages. This together 
    with the above discussion proves Theorem~\ref{cirr}. 
\end{proof}

\subsection{
    Generalized simple intersections. Proof of Theorems ~\ref{doubint}, ~\ref{irred1}, \ref{intch} and Corollary  \ref{cdoubint}
}

\begin{proof} {\bf of Theorem \ref{doubint}.} 
    The first statement of the theorem follows from \cite[appendix C]{bt0} 
    and~\cite[theorem 1.12]{g18}, and the latter theorem also states that a 
     semi-infinite interval  $Sl\subset\Lambda_l^+$ lying in the upper half-plane and 
     bounded by the highest  
    generalized simple intersection in $\La_l^+$ lies entirely in $L_l$. 
    Let $s\in\nn$, $s<l$. 
    The phase-lock area $L_s$ lies on the left from the phase-lock area $L_l$. 
    Hence, its upper part consisting of the points with  ordinates $A$ large enough 
    lies on the left from  the above interval $Sl$. On the other hand, its 
     growth point, the intersection $L_s\cap\{ A=0\}$, 
    has abscissa $B=\sqrt{s^2\omega^2+1}$, see \cite[corollary 3]{buch1}. 
    Hence, it lies on the right from the axis 
    $\La_l=\{B=l\omega\}$, whenever $\omega$ is small enough. Finally, 
    the upper part of each boundary component,  
    $\partial L_{s,\pm}\cap\{ A\geq0\}$,  intersects $\Lambda_l^+$, whenever 
    $\omega$ is small enough, since  it contains  a point on the left from the 
    semiaxis $\Lambda_l^+$ and the above growth point lying on its right. For $s=0$ 
    the intersection point of the boundary $\partial L_0$ with the positive 
    $B$-semiaxis is the point $(1,0)$. This follows, e.\,g., from arguments 
    in~\cite[example~5.23]{bg2}. The corresponding differential 
    equation~\eqref{jostor} has the form $\dot\varphi=\frac{1-\sin\varphi}{\omega}$; 
    $l=\frac1\omega$, $\mu=0$. 
    Hence, the Poincar\'e  map of the corresponding vector field~\eqref{josvec} 
    is parabolic with fixed point $\frac{\pi}2(\modulo2\pi\zz)$, and thus, 
    $(1,0)\in\partial L_{0,+}$. For every $l\in\nn$ and every $\omega>0$ small 
    enough (dependently on $l$) the latter point $(1,0)$ lies on the right from the axis 
    $\Lambda_l=\{ B=l\omega\}$. Hence, $\partial L_{0,+}$ intersects $\Lambda_l^+$, 
    as in the above discussion. This proves Theorem \ref{doubint}.
\end{proof}

\begin{proof}  {\bf of Corollary \ref{cdoubint}.} The number of generalized simple intersections lying in a given semiaxis 
$\La_l^+$ is 
no greater than $l$, since they are defined by a polynomial in $(\la,\mu^2)$ of degree $l$, 
$\mu=\frac A{2\omega}$.  For small $\omega$ the number of generalized simple intersections in $\La_l^+$ 
mentioned in Theorem \ref{doubint}  
is at least $l$: at least one lies in $\partial L_l$; 
at least two lie in $\partial L_s$ for $0<s<l$, $s\equiv l(\modulo 2)$; 
at least one lies in $\partial L_0$, if $l$ is even. Therefore, their number is exactly 
equal to $l$, exactly one of them lies in $\partial L_l$, 
and they depend analytically on $\omega$, since the squares 
of their $\mu$-coordinates are distinct roots 
of an analytic family of polynomials of degree $l$. Their abscissas are equal to the same number 
$\omega l$, which tends to 0, as $l\to0$. This proves Corollary \ref{cdoubint}.
\end{proof}

Recall that we denote $r=\frac1{2\omega}$. The mapping 
$\pi\colon\cc^2\to\cc^2$, $\pi(\mu,r)=(\la=r^2-\mu^2, \mu)$, sends the $(\mu,r)$-coordinates 
of the generalized  simple intersections in $\La_l(\omega)$ to the zero locus 
$\Gamma_l=\{ P_l(\la,\mu^2)=0\}$ and vice versa: for every  
real point  $(\la,\mu)\in\Gamma_l$ with $\la+\mu^2>0$  its preimage with $r>0$ 
is a generalized simple intersection 
(Theorems~\ref{bt02} and~\ref{tpol}). Set $\widehat\Gamma_l=\pi^{-1}(\Gamma_l)$. 
One has $\hsi_l\subset\widehat\Gamma_l$  and $\pi(\hsi_l)=\Gamma_l$, by the above 
discussion, and since the image of the algebraic set $\hsi_l$ under 
degree two rational branched cover $\pi$ is algebraic (Remmert 
Proper Mapping Theorem and Chow Theorem \cite{griff}). 
We already know that the complex 
algebraic curve $\Gamma_l$ is irreducible (Theorem~\ref{cirr}). This proves 
the first statement of Theorem \ref{irred1}. The mapping 
$\pi$ has degree two. Therefore, the set $\widehat\Gamma_l$, which 
contains $\hsi_l$, is an algebraic curve that either is 
irreducible, or has two irreducible components. We already know that 
$\hsi_l=\hsi_{l,+}\cup\hsi_{l,-}$. Let us show that the sets $\hsi_{l,\pm}$ are 
non-empty analytic curves. This will imply their irreducibility, the 
equality $\widehat\Gamma_l=\hsi_l$ and Theorem~\ref{irred1}.

\begin{lemma}
    \label{2irr}
    For every $l\in\nn$, every $\omega>0$ small enough and every sign $\pm$ 
    there exists an $s\in\zz_{\geqslant 0}$, $s\equiv l(\modulo 2)$, $s\leqslant l$, such 
    that the boundary component $\partial L_{s,\pm}$ intersects $\Lambda_l$.
\end{lemma}

    For $l\geqslant3$ the statement of the lemma holds for every 
    $s\equiv l(\modulo 2)$, $0<s<l$, by Theorem~\ref{doubint}. 
    Below we prove it for $l=1,2$. In the proofs of Lemma \ref{2irr} 
    and Theorems \ref{irred1}, \ref{intch} we use the  following 
    proposition. 

    \begin{proposition}
        \label{halfp} \cite[p. 40, the discussion around formulas (10)-(11)]{krs}
        For every $s\in\zz$ and every point $(B,A)\in\partial L_{s,\pm}$ 
        the orbit of the corresponding vector field (\ref{josvec}) starting at 
        $(\varphi,\tau)=(\pm\frac{\pi}2,0)$ is $2\pi$-periodic and invariant under the symmetry 
        $I\colon(\varphi,\tau)\mapsto(\pi-\varphi,-\tau)$ (found by A.\,V.\,Klimenko 
        in~\cite{RK}) of the field~\eqref{josvec}. In particular, the value $(\modulo2\pi)$ 
        of the corresponding solution of differential equation on a function $\varphi(\tau)$ 
        at the half-period $\tau=\pi$ is the initial value $\pm\frac{\pi}2$, if $s$ is even, 
        and $\mp\frac{\pi}2$, if $s$ is odd.
    \end{proposition}

    \begin{proof} The orbit under question is $2\pi$-periodic, since 
    $(B,A)\in\partial L_{s,\pm}$, and hence, $\pm\frac{\pi}2$ is a fixed point of the Poincar\'e 
    map of the field (\ref{josvec}). On the other hand, the points $\pm\frac{\pi}2$ are the 
    only  fixed points of the circle involution $\varphi\mapsto\pi-\varphi$, and $0$, $\pi$ are the 
        only fixed points of the circle involution $\tau\mapsto-\tau$. These statements imply 
        the proposition.
    \end{proof}

    \begin{proof} {\bf of Theorem \ref{intch}.} 
        The transformation 
        $J\colon(\varphi,\tau)\mapsto(\varphi,\tau+\pi)$ sends a vector field~\eqref{josvec} corresponding to a point $(B,A)$ 
        to the similar vector field denoted \eqref{josvec}$^*$ 
        with changed sign at $\mu$, i.\,e., corresponding to 
        the point $(B,-A)$, see~\cite[p.76]{bg2}. The rotation number obviously remains the same; let us denote it by $s$. 
        The symmetry $J$ transforms an orbit $O_1$ of the field 
        \eqref{josvec} through 
        the point $(\pm\frac{\pi}2,0)$ to an orbit $O_2$  
        of the field \eqref{josvec}$^*$ through the point 
        $(\pm\frac{\pi}2,\pi)$. The orbit $O_2$  is $2\pi$-periodic, if and only 
        if so is $O_1$. Fix a sign $\pm$, and let $O_1$ be periodic; 
        then $s\in\zz$. In the case, 
        when  $s$ is even, the orbit $O_1$ passes through both points 
        $(\pm\frac{\pi}2,0)$ and $(\pm\frac{\pi}2,\pi)$, by the last statement of Proposition~\ref{halfp}. 
        Hence, both of them are fixed by the time $2\pi$ flow map of the field. 
        This together with the above discussion implies that the orbit $O_2$ is periodic and 
        also passes through the same two points. Finally, both  vector fields~\eqref{josvec} 
        corresponding to the points $(B,A)$ and $(B,-A)$ have the same periodic point 
        $(\pm\frac{\pi}2,0)$ and the same rotation number. This implies that the symmetry 
        $(B,A)\mapsto(B,-A)$ leaves each curve $\partial L_{s,\pm}$ invariant. Similarly in the 
        case, when $s$ is odd, the orbit $O_1$ starting at $(\pm\frac{\pi}2,0)$ 
        passes through  the point 
        $(\mp\frac{\pi}2,\pi)$, by Proposition \ref{halfp}. Hence, the orbit $O_2=J(O_1)$ 
        passes through the point $(\mp\frac{\pi}2,0)$ and is periodic. This implies that 
        the symmetry $(B,A)\mapsto(B,-A)$ interchanges the curves $\partial L_{s,\pm}$. 
        
        Now let us prove the last statement 
       of Theorem \ref{intch}, on the symmetry $(B,A)\mapsto(-B,A)$. 
       The transformation $(\varphi,\tau)\mapsto(-\varphi,\tau+\pi)$ 
      sends a vector field (\ref{josvec}) corresponding to a point 
      $(B,A)$ to the same field with 
       opposite sign at $l$, i.e., corresponding to $(-B,A)$; 
       the latter field will be denoted by  \eqref{josvec}$^*$. It changes the sign of the rotation number. It transforms an orbit $O_1$ 
       of the field \eqref{josvec} through 
        the point $(\pm\frac{\pi}2,0)$ to an orbit $O_2$ of the field  
        \eqref{josvec}$^*$ through the point 
        $(\mp\frac{\pi}2,\pi)$. Fix a sign $\pm$, and let $O_1$ be periodic. Let $s$ denote its rotation number; then $s\in\zz$. 
        In the case, 
        when $s$ is even, the orbit $O_1$ passes through both points 
        $(\pm\frac{\pi}2,0)$ and $(\pm\frac{\pi}2,\pi)$; the orbit $O_2$ is periodic and 
         passes through the  two points $(\mp\frac{\pi}2,0)$ and $(\mp\frac{\pi}2,\pi)$, as in the above discussion. Finally, the vector fields~\eqref{josvec} 
        corresponding to the points $(B,A)$ and $(-B,A)$ have 
         periodic points 
        $(\pm\frac{\pi}2,0)$ and $(\mp\frac{\pi}2,0)$ and opposite 
         rotation numbers. Thus, the symmetry 
         $(B,A)\mapsto(-B,A)$ interchanges the curves 
         $\partial L_{s,\pm}$ and $\partial L_{-s,\mp}$.  In the case, when $s$ is odd, the orbit $O_1$ passes through the points 
        $(\pm\frac{\pi}2,0)$ and $(\mp\frac{\pi}2,\pi)$; the orbit $O_2$ is periodic and 
         passes through the  two points $(\mp\frac{\pi}2,\pi)$ and 
         $(\pm\frac{\pi}2,0)$. Hence, the vector fields~\eqref{josvec} 
        corresponding to the points $(B,A)$ and $(-B,A)$ have 
      the same   periodic point 
        $(\pm\frac{\pi}2,0)$  and opposite rotation numbers. Thus, the symmetry interchanges the curves 
         $\partial L_{s,\pm}$ and $\partial L_{-s,\pm}$
        This proves Theorem \ref{intch}.
  \end{proof}

    \begin{proof} {\bf of Lemma \ref{2irr}.} Recall that it suffices to 
    prove Lemma \ref{2irr} for $l=1,2$ (Theorem \ref{doubint}). 
     Let $l=1$. 
    At least one boundary component $\partial L_{1,\pm}$ intersects $\Lambda_1$ (Theorem~
    \ref{doubint}). The other component  is symmetric to it with respect to the 
    $B$-axis (Theorem~\ref{intch}), and hence, also intersects $\Lambda_1$. 
    This proves the statement of Lemma~\ref{2irr} for $l=1$. The case, when $l=2$, 
    was already treated in~\cite[appendix~A,~subsection~A.1]{bt0}. In more detail, 
    let~\eqref{jostor}${}^*$ and~\eqref{josvec}${}^*$ denote respectively the 
    differential equation~\eqref{jostor} and the vector field~\eqref{josvec} 
    where the parameter $l$ is taken with opposite sign. It was shown 
    in loc.\,cit.\,that for every positive $\omega<\frac12$ there are two 
    positive values $\mu_{\pm}=(2\omega)^{-1}\sqrt{1\mp2\omega}$ of the parameter $\mu$ for 
    which the corresponding equation (\ref{heun}) has a polynomial solution; 
    these polynomial solutions induce solutions $\varphi(\tau)$ of 
    equation~\eqref{jostor}${}^*$ with $\varphi(0)=\mp\frac{\pi}2$, 
    $\varphi(2\pi)=\varphi(0)-2\pi\pm2\pi$. Hence, the corresponding trajectories of 
    dynamical systems~\eqref{josvec}${}^*$ on torus are periodic with initial 
    conditions $\mp\frac{\pi}2(\modulo2\pi)$ and rotation numbers 0 (for $\mu_+$) and $-2$ 
    (for $\mu_-$). 
    
    Set $B=2\omega$, $A_{\pm}=2\mu_{\pm}\omega$.     
    By construction, $(B,A_{\pm})\in\La_2$ are generalized simple 
    intersections, and the above  system (\ref{josvec})$^*$ 
    is the system (\ref{josvec}) 
    corresponding to the parameters 
    $(-B,A_{\pm})$.    The above discussion implies that 
     $(-B,A_+)\in \partial L_{0,-}$ and $(-B,A_-)\in \partial L_{-2,+}$. Therefore, 
     $(B,A_+)\in \partial L_{0,+}\cap\La_2$, $(B,A_-)\in \partial 
     L_{2,-}\cap\La_2$, by Theorem \ref{intch}. This proves Lemma~\ref{2irr}.
\end{proof}

Lemma \ref{2irr} together with the discussion before it imply that the algebraic 
set $\hsi_l$, which consists of at most two irreducible components, is the union 
of two non-trivial algebraic curves $\hsi_{l,\pm}$. Therefore, the latter curves 
are irreducible and the number of components under question is exactly two. 
This proves Theorem \ref{irred1}.

\subsection{
    Boundaries of phase-lock areas. Proof of Theorem \ref{irred2}.
}

Recall that the complexification of equation~\eqref{jostor} is the corresponding 
Riccati equation~\eqref{ricc}. For every collection of complex parameters $l$, 
$\mu$, $\omega$ in~\eqref{ricc} with $\omega\neq0$ the {\it monodromy transformation} $\mon$ of the 
Riccati equation (\ref{ricc}) acts on the space of initial conditions at $z=1$ 
by analytic continuation of solutions along the positively oriented unit circle 
in the $z$-line.  It is a M\"obius transformation of the Riemann sphere, since 
the equation under question is the projectivization of a linear equation. 
For real parameter values  the corresponding Poincar\'e 
map $h_{(B,A)}$ acts on the circle $\rr\slash2\pi\zz$ with coordinate $\varphi$, 
and the variable change $\varphi\mapsto\Phi=e^{i\varphi}$ sends it to the unit circle 
in the Riemann sphere. The  monodromy $\mon$ is the complexification of the 
Poincar\'e map considered as a unit circle diffeomorhism. Thus,  the Poincar\'e 
map is a {\it M\"obius circle diffeomorphism:} a unit circle diffeomorphism  
 that extends as a M\"obius transformation of the Riemann sphere. 

Recall that the condition 
saying that $h_{(B,A)}$ is parabolic implies that it fixes some of the points 
$\pm \frac{\pi}2(\modulo2\pi\zz)$, see~Remark~\ref{rklim}. This implies the 
following corollary. 

\begin{corollary}
    \label{cormon}
    For every $\omega>0$, $s\in\nn$ and  sign $\pm$ if a point 
    $(B,A)\in\rr^2$ lies in $\partial L_{s,\pm}$, then the monodromy $\mon$ of 
    the corresponding Riccati equation (\ref{ricc}) fixes the point $\pm i$. 
    Thus, the subsets $\mcl_{\pm}^{odd (even)}$, and hence, their 
    complexifications $\hmcl_{\pm}^{odd (even)}$ lie in the hypersurfaces 
    \[
        \Sigma_{\pm}\subset\cc^3_{(B,A,r)}, \ r=\frac1{2\omega},
    \]
    of those parameter values, for which the monodromy $\mon$ fixes the point $\pm i$. 
\end{corollary}

For every $l\in\zz$ set 
\[
    \mcl_{l,\pm} =
    \bigcup_{\omega>0}\left(\partial L_{l,\pm}(\omega)\times\left\{\frac1{2\omega}\right\}\right)\subset 
    \rr^3_{(B,A,r)},
\]
\[
    \hmcl_{l,\pm}:=\text{ the minimal analytic subset in } \cc^3 
    \text{ containing } \mcl_{l,\pm}.
\]
Recall that for every $l\in\nn$ by $\si_{l,+}$ we denote the union of families of 
tuples $(\mu, r=\frac1{2\omega})\in\rr^2$ that correspond to the generalized 
simple intersections with $\omega=\frac1{2r}$, $B=l\omega$, $A=2\mu\omega$. 
Consider the transformation 
\[
    g_l:\cc^2_{(\mu,r)}\mapsto\cc^3_{(B,A,r)}: \ g_l(\mu,r)=\left(\frac l{2r}, \frac{\mu}r, r\right).
\]
Set 
\[
    \wt\Gamma_{l,\pm}:= g_l(\hsi_{l,\pm}), \ \widehat{\Gamma}_{l,\pm}:= 
    \text{ the closure of the curve } \wt\Gamma_{l,\pm} \text{ in the usual 
topology.}
\]
\def\whg{\widehat{\Gamma}}
It is clear that  $\widehat{\Gamma}_{l,\pm}$ is an irreducible algebraic curve, since so is 
$\hsi_{l,\pm}$ (Theorem~\ref{irred1}) and the mapping $g_l$ is rational and injective on the set $r\neq0$. 

\begin{proposition} \label{pmono} Fix an $\omega>0$.  Consider the family of the 
    Poincar\'e maps $h_{(B,A)}$ as a family of elements of  M\"obius circle 
    transformation group. Let $(B_0,A_0)\in\rr^2$ lie in the boundary curve 
    $\partial L_{s,\pm}$ of a phase-lock area $L_s$. The derivative in $B$ of 
    the Poincar\'e map family $h_{(B,A)}$ at $(B_0,A_0)$ is a vector transversal to the hypersurface 
    of M\"obius circle transformations fixing $\pm \frac{\pi}2(\modulo2\pi\zz)$. In particular $\mcl_{s,\pm}$ is a smooth 
    two-dimensional surface.
    \end{proposition}

The proposition follows from monotonicity of the Poincar\'e map as a function 
of the parameter $B$ (classical). 

\begin{proposition}
    \label{pris}
    Each complex analytic set $\hmcl_{l,\pm}$ is irreducible and purely 
    two-dimensional. For every $s\in\nn$, $l>s$, $l\equiv s(\modulo2\zz)$ and sign $\pm$ 
    the surface $\hmcl_{s,\pm}$ contains the curve $\widehat\Gamma_{l,\pm}$. 
    The same statement also holds for $s=0$, the sign $+$ and even $l\in\nn$.
\end{proposition}

\begin{proof} The minimal complex analytic subset $\widehat\mcl_{s,\pm}$ 
     containing 
    $\mcl_{s,\pm}$ is irreducible and at least two-dimensional, since 
    $\mcl_{s,\pm}$ is smooth and two-dimensional, by the last statement of Proposition \ref{pmono}. It cannot have 
    bigger dimension, since it is already contained in a two-dimensional 
    analytic subset $\Sigma_{\pm}$. This implies that it is purely 
    two-dimensional. Let us prove the inclusion (the second and third statements) of 
    Proposition~\ref{pris}. The surface $\mcl_{s,\pm}$ contains a family of 
    simple intersections in $\Lambda_l(\omega)$ with small $\omega$
     (Theorem~\ref{doubint}). The minimal complex analytic subset in $\cc^3$ 
    containing the above family of simple intersections is the curve 
    $\widehat\Gamma_{l,\pm}$, by definition and  irreducibility (Theorem~\ref{irred1}). 
    This implies that the irreducible complex surface $\hmcl_{s,\pm}$ contains 
    the irreducible curve $\widehat\Gamma_{l,\pm}$ and proves Proposition~\ref{pris}.
\end{proof}

\begin{proposition}
    \label{proreg}
    The analytic surface $\Sigma_{\pm}$ is regular at each its 
    real point corresponding to a simple intersection.
\end{proposition}

Proposition~\ref{proreg} follows from Proposition~\ref{pmono}.

\begin{proposition}
    \label{coincide}
    For every sign $\pm$ and every $s_1,s_2\in\nn$, $s_1-s_2$ being even, one 
    has $\hmcl_{s_1,\pm}=\hmcl_{s_2,\pm}$. The same statement holds for the 
    sign~"$+$" in the case, when $s_{1,2}\in\zz_{\geqslant0}$ and $s_i=0$ for some $i$.  For every sign $\pm$ and every $s_1,s_2\in\zz_{<0}$, $s_1-s_2$ being even, one 
    has $\hmcl_{s_1,\pm}=\hmcl_{s_2,\pm}$. The same statement holds for the 
    sign~"$-$" in the case, when $s_{1,2}\in\zz_{\leqslant0}$ and $s_i=0$ for some $i$. 
\end{proposition}

\begin{proof} Let us prove only the  statements of the 
proposition for $s_1,s_2\geqslant0$; its  statements for $s_1,s_2\leqslant0$ 
then 
follow by symmetry and Theorem \ref{intch}. 

    The surfaces $\hmcl_{s_i,\pm}$ contain the irreducible analytic curve 
    $\widehat\Gamma_{l,\pm}$, whenever $l$ is greater than each $s_i$ and has the same 
    parity (Proposition~\ref{pris}). Their germs at each point  $p\in\wt\Gamma_{l,\pm}$ 
    representing a generalized simple intersection coincide with the  germ of the 
    ambient  surface $\Sigma_{\pm}\supset\hmcl_{s_i,\pm}$, since the latter germ is 
    regular and  two-dimensional (Proposition \ref{proreg}). Finally, the surfaces 
    $\hmcl_{s_i,\pm}$, $i=1,2$, are irreducible and have the same germ at $p$. 
    Therefore, they coincide. This proves Proposition~\ref{coincide}.
\end{proof} 

\begin{proposition}
    \label{co+-}
    For every $s\in\nn$ one has $\hmcl_{s,\pm}=\hmcl_{-s,\pm}$. 
\end{proposition}

\begin{proof}
    Recall that the growth point of the phase-lock area $L_s$, $s\in\nn$,~i.\,e., 
    its intersection with the horizontal $B$-axis has abscissa $B_s(\omega)\sign(s)$, 
    \[
        B_s(\omega):=\sqrt{s^2\omega^2+1},
    \]
    see \cite[corollary 3]{buch1}. Therefore, the corresponding complex surface $\hmcl_{s,\pm}$ contains 
    the algebraic curve  
    \[
        G_s:=\{ A=0, \ B^2=s^2\omega^2+1\}\subset\cc^3_{(B,A,\frac1{2\omega})}.
    \]
    The curve $G_s$ is obviously irreducible and contained in both surfaces 
    $\hmcl_{s,\pm}$ and $\hmcl_{-s,\pm}$. For every given $\omega>0$ the curve $G_s$ 
    contains two real symmetric points $q_+$ and $q_-$ with coordinates 
    $B= B_s(\omega)$ and $B=-B_s(\omega)$. The germ of the ambient surface 
    $\Sigma_{\pm}$ at a real point in $G_s$ (i.\,e., at a point corresponding 
    to a growth point) is regular, as in Proposition~\ref{proreg}. This implies 
    that the germ of the surface $\hmcl_{s,\pm}$ at $q_+$ and its germ obtained 
    by analytic extension to $q_-$ along a path in $G_s$ coincide with the germs 
    of the ambient surface $\Sigma_{\pm}$, as in the proof of 
    Proposition~\ref{coincide}. Similar statement should hold for the surface 
    $\hmcl_{-s,\pm}$. Finally, the irreducible surfaces $\hmcl_{s,\pm}$ and 
    $\hmcl_{-s,\pm}$ have coinciding germs at the point $q_-$. Hence, they 
    coincide. This proves the proposition.
\end{proof}

\begin{proposition}
    \label{irrtot}
    The surfaces $\hmcl_{\pm}^{even (odd)}$ are irreducible.
\end{proposition}

\begin{proof}
    For every $s_1,s_2\in\zz$ having the same parity the minimal analytic sets $\hmcl_{s_j,\pm}$ 
    containing the real surfaces $\mcl_{s_j,\pm}$, $j=1,2$, coincide, 
    by Propositions ~\ref{coincide} and~\ref{co+-}. 
    This implies the statement of Proposition \ref{irrtot}.
\end{proof}

\begin{proposition}
    \label{disteo}
    The four irreducible surfaces $\hmcl_{\pm}^{even (odd)}$ are distinct. 
\end{proposition}

\begin{proof}
    It is clear that the above surfaces with different signs indices should be 
    distinct, as are the ambient surfaces $\Sigma_{\pm}$. Let us prove that 
    $\hmcl_{\pm}^{even}\neq\hmcl_{\pm}^{odd}$. 

    For every point $q\in\Sigma_{\pm}$ let $f_q(z)$ denote the corresponding 
    solution of the Riccati equation (\ref{ricc}) with 
    $f_q(1)=\pm i$. It is  meromorphic on $\cc^*$, since the monodromy fixes the 
    point $\pm i$ and hence, fixes its initial branch at $z=1$. 
    In the case, when the function $f_q$ has neither zeros, nor 
    poles in the unit circle $S^1=\{ |z|=1\}$, let $\rho_q$ denote the index 
    of the restriction $f_q: S^1\to\cc^*$: the argument increment divided by $2\pi$. 

    \begin{proposition}
        \label{ppar} 
        Each surface $\hmcl_{\pm}^{even (odd)}$ contains a nowhere 
        dense real-analytic subset $\mathcal S$ consisting of those points $q$ 
        for which the function $f_q(z)$ has either zeros, or poles in $S^1$. 
        For $q\in\hmcl_{\pm}^{even (odd)}\setminus\mathcal S$ the indices $\rho_q$ are even (respectively, odd). 
    \end{proposition}

    \begin{proof}
        In the case, when $q$ is a real point, the function $f_q$ takes values 
        in the unit circle on $S^1$ (thus, has neither zeros, nor poles there), 
        and  the number $\rho_q$ is clearly equal to the rotation number of the 
        corresponding dynamical system~\eqref{josvec} on torus. This together with 
        irreducibility implies 
        the first statement of the proposition. For every $q\in\mathcal S$ if 
        the function $f_q(z)$ has a zero at a point $z$ of the unit circle, 
        then it has a pole at its conjugate $z^{-1}=\bar z$, and the 
        multiplicities of zero and pole are equal. This follows from the 
        symmetry of the Riccati equation~\eqref{ricc} under the involution 
        $(\Phi,z)\mapsto(-\Phi^{-1},z^{-1})$, which is the complexification of 
        the symmetry $(\varphi,\tau)\mapsto(\pi-\varphi,-\tau)$ of equation~\eqref{jostor}. 
        The above statement together with symmetry imply that when $q$ crosses 
        the real codimension one hypersurface $\mathcal S$, the parity of the 
        index $\rho_q$ does not change. This proves the proposition.
    \end{proof}

    Proposition~\ref{ppar} implies Proposition~\ref{disteo}.
\end{proof}

Proposition \ref{disteo}~implies Theorem~\ref{irred2}. 

\section{
    Spectral curves: genera and real parts 
 (I.\,V.\,Netay)
 }\label{genera}

Let us consider some other curves birationally equivalent to~$\Gamma_l$ but
more convenient. Let us introduce an $l\times l$-matrix
\[
	\mathcal{G}_l =
	\begin{pmatrix}
		0 & \ldots & 0 & \mu \\
		\vdots && \rotatebox{75}{$\ddots$} & -(l-1) \\
		0 &\rotatebox{75}{$\ddots$} & \rotatebox{75}{$\ddots$} & \\
		\mu & -1& & 0\\
	\end{pmatrix}.
\]

The following relation (found in~\cite[\S3,~eq.~(30)]{bt0}) holds:
\[
    H_l + (r^2 - \mu^2)\operatorname{Id} =
    -(\mathcal{G}_l + r\operatorname{Id})
     (\mathcal{G}_l - r\operatorname{Id})
\]
for a variable~$r$. Let us put~$\lambda = r^2 - \mu^2$.
Then we obtain
\[
    \det(H_l + \lambda \operatorname{Id}) = (-1)^{l} 
    \det(\mathcal{G}_l + r\operatorname{Id}) 
    \det(\mathcal{G}_l - r\operatorname{Id}).
\]
Then the mapping
\[
    \pi\colon (\mu,r) \mapsto (\lambda = r^2 - \mu^2, \mu)
\]
maps points of the curve~$\{
    \det(\mathcal{G}_l + r\operatorname{Id}) 
    \det(\mathcal{G}_l - r\operatorname{Id}) = 0
\}$ of affine plane~$\mathbb{A}_{\mu,r}$ to points 
of the curve~$\Gamma_l=\{\det(H_l + \lambda\operatorname{Id})=0\}\subset\mathbb A_{\la,\mu}$.

Denote the curves~$\{\det(\mathcal{G}_l \pm r\operatorname{Id})=0\}$ by~$\Xi_l^\pm$. 
 Then $\pi^{-1}(\Gamma_l)$ is the union of the curves $\Xi_l^{\pm}$. 
On the other hand, the latter preimage  is the union of two irreducible 
curves $\hsi_{l,\pm}$  (Theorem~\ref{irred1}). 
Hence, the curves $\Xi_{l}^{\pm}$ are irreducible and coincide with the curves 
$\hsi_{l,\pm}$ (up to transposition, i.\,e., sign change), and 
   ~$\pi$ maps each of~$\Xi_l^-$ and~$\Xi_l^+$ birationally to~$\Gamma_l$. 

Let us look at the projective closures $\overline{\Xi_l^{\pm}}$ of some of the curves~$\Xi_l^\pm$:
\begin{itemize}
    \item[$l=1$]: the curve is a line;
    \item[$l=2$]: the curve is a smooth conic and is rational;
    \item[$l=3$]: the curve is a singular cubic and therefore is rational;
    \item[$l=4$]: it is a quartic with two simple self-intersections, and hence,  an elliptic curve.
\end{itemize}

Consider the following diagram of morphisms:
\[
    \xymatrix{
        \widehat{\Xi}_l^\pm \ar[d]\ar@{^{(}->}[r] &
        \mathbb{P}^1\times\mathbb{P}^1 & 
        X_7 \ar[l]_-{\operatorname{Bl}_{\{\theta=0\}}}\ar@/_1pt/[d]^{\operatorname{Bl}_{p_\pm}}
        \\
        \Xi_l^\pm \ar[d]_-{2:1}\ar@{^{(}->}[r] &
        \mathbb{A}_{\mu,r}^2 \ar@{->>}[d]|- {\ \ \pi\colon}^-{\ \  r\mapsto\lambda=r^2-\mu^2}\ar@{->}[u]\ar@{^{(}->}[r]&
        \mathbb{P}_{\mu:r:\theta}^2 \ar@{-->}[d]^{\widetilde{\pi}}\ar@{-->}[ul]
        \\
        \Gamma_l \ar@{^{(}->}[r] &
        \mathbb{A}_{\lambda,\mu}^2 \ar@{^{(}->}[r]&
        \mathbb{P}_{\lambda:\mu:\zeta}^2 &
        \\
    }
\]
Let us describe the diagram. 
The natural compactification of the curve~$\Gamma_l$ lies 
in~$\mathbb{P}_{\lambda:\mu:\zeta}^2$.
At the same time the natural compactification of the curve~$\Xi_l^\pm$ lies in~$\mathbb{P}_{\mu:r:\theta}^2$.
The rational map~$\wt\pi:\xymatrix{\mathbb{P}_{\mu:r:\theta}^2 \ar@{-->}[r] & \mathbb{P}_{\lambda:\mu:\zeta}^2}$
is defined as the map $\pi$ on an open subset, the affine chart~$\mathbb{A}_{\mu,r}^2=\cc^2_{(\mu,r)}$, and has degree~$2$.
It is easy to define the map in homogeneous coordinates:
\[
 \widetilde\pi:   \begin{cases}
        \lambda & \mapsto r^2 - \mu^2, \\
        \mu & \mapsto \mu\theta, \\
        \theta & \mapsto \theta^2.\\
    \end{cases}
\]
The rational mapping $\widetilde\pi$ is not well-defined exactly 
at two points, namely,~$(\mu:r:\theta) = (1:\pm1:0)$. Let us denote these points
by~$p_\pm$ correspondingly.

If we blow them up, we get del Pezzo surface $X_7$ of degree~$7$. It has three
$(-1)$-curves: preimages of points $p_\pm$ (which will be denoted by $C_{\pm}$) 
 and the strict transform $C_{\infty}$ of the line
$\{\theta=0\}$ passing through them. If we blow $C_{\infty}$ down, 
then we get del Pezzo surface of degree~$8$, namely,~$\mathbb{P}^1\times\mathbb{P}^1$.
Consider the composition $\mathbb{A}_{\mu,r}^2 \to \mathbb{P}^1\times\mathbb{P}^1$ of 
the inclusion $\mathbb{A}_{\mu,r}^2\to\mathbb P^2_{\mu:r:\theta}$ with the above blows up and down. It  
is well-defined, because the points blown up lie at  infinity. Preimages of the 
two line families~$\{\mathrm{pt}\}\times\mathbb{P}^1$ and~$\mathbb{P}^1\times\{\mathrm{pt}\}$
are lines~$r\pm\mu=\operatorname{const}$.

\begin{proposition} \label{prosi}
    The only real singular points of projective closures~$\overline{\Xi_l^\pm}$ are~$(1:\pm1:0)$
    (except smooth curves for $l=1,2$ and the smooth point $p_+$ for $\Xi_l^+$ 
    and $p_-$ for $\Xi_l^-$ for $l=3$).
    In particular, the affine curves~$\Xi_l^\pm(\mathbb{R})$ are smooth.
\end{proposition}

\begin{proof}
    Absence of real singularities except for those lying in the infinity line and the 
    line~$\{\mu=0\}$ is guarantied by
    smoothness of~$\Gamma_l$ outside the line~$\{\mu=0\}$ (because the image of a singular point 
    is singular). Also smoothness of~$\Xi_l^-\cup\Xi_l^+$ outside the line~$\{\mu=0\}$ is
    proved in~\cite[\S3,~Cor.\,4]{bt0}.
    The only points of~$\Xi^\pm$ at infinity are~$(1:\pm1:0)$.
    The multiplicities of these two points are calculated in the following 
    proposition.

    It remains only to consider the line~$\{\mu=0\}$.
    For $\mu=0$ one has 
    \begin{equation}
        \det(\mathcal{G}_n+r\operatorname{Id}) = \begin{cases} r\prod_{k=1}^{\frac{l-1}2}(r^2-k(l-k)) & \text{ if } l \text { is odd,}\\
        r(r-\frac{l}2)\prod_{k=1}^{\frac {l-2}2}(r^2-k(l-k)) & \text{ if } l \text { is even,}
        \end{cases}
    \label{gnr}\end{equation}
    since appropriate permuting of lines and columns of the matrix in the above 
    left-hand side 
    makes it block-diagonal with obvious blocks of dimension two or one. The above polynomial has no multiple roots. Therefore all the intersections of a given 
    curve $\Xi^{\pm}$ 
    with the line~$\{\mu=0\}$ are simple. In particular, this implies that there are no
    singular points among the intersections.
\end{proof}

\begin{definition}
    Let us denote the {\it multiplicity} of a point~$P$ on a curve~$C$ by~$\upmu_P(C)$: 
    this is its intersection index at $P$ with a generic line through $P$. Recall that 
    in the case when $P$ is an intersection of $k$ pairwise transversal smooth 
    local branches, its multiplicity is equal to $k$; in this case we will call $P$ a {\it pairwise 
    transversal  self-intersection}. A regular point, which corresponds to one branch, 
    will be also treated as a pairwise transversal self-intersection.
\end{definition}

\begin{proposition} \label{trpp} Each of the points $p_+$, $p_-$ 
 is a pairwise transversal 
 self-intersection  
for every curve $\overline{\Xi^{\pm}_l}$, and one has 
    $\upmu_{p_\pm}(\overline{\Xi_l^\pm}) = \left[\frac{l}{2}\right]$;
    $\upmu_{p_\pm}(\overline{\Xi_l^\mp}) = l-\left[\frac{l}{2}\right]$. 
    Moreover, all the local branches of the 
    curve $\overline{\Xi^+_l}\cup\overline{\Xi^-_l}$ at $p_{\pm}$ are transversal 
    to the infinity line $\{\theta=0\}$.
\end{proposition}

\begin{proof} Consider a germ of analytic curve $C$ at $O$ in an affine chart. Recall that 
to show that $O$ is a pairwise transversal self-intersection of multiplicity $k$, it suffices to prove that 
the function $f$ defining $C$ (i.e., generating the ideal of germs of functions vanishing 
at $C$) has lower Taylor homogeneous part of degree $k$ that is a product of 
$k$ pairwise non-proportional 
linear forms. In the case, when $O=p_{\pm}$, to show that each local branch is transversal to the infinity line, one has to show that 
the above lower homogeneous part is not divisible by $\theta$. 

    The symmetry~$r\leftrightarrow -r$ swaps pairs~$(p_+,p_-)$ 
    and~$(\overline{\Xi_l^+},\overline{\Xi_l^-})$ for each~$l$.
    So let us consider only the curve~$\overline{\Xi_l^+}$.

    Let us consider the affine chart~$\mu=1$.
    At first, consider the point~$p_+=(1:1:0)$. So, we substitute~$r=1+a$ with $a$ 
    being small.

    Suppose~$l$ is even.
    In the new coordinates the curve~$\Xi_l^+$ is defined by
    \begin{gather*}
        \begin{vmatrix}
            1 + a & 0 & \cdots& \cdots & 0 & 1 \\
            0 & \ddots & & & \udots & -(l-1)\theta \\
            \vdots & & 1+a & 1 & \udots & 0 \\
            \vdots & & 1 & 1+a-\frac{l}{2}\theta & & \vdots \\
            0 & \udots & \udots & & \ddots & 0 \\
            1 & -\theta & 0 & \cdots & 0 & 1 + a \\
        \end{vmatrix} =
        \intertext{Let us subtract columns~$k$ from~$l-k$ for~$k=1,\ldots,l/2$:}
        =
        \begin{vmatrix}
            1 + a & 0 & \cdots& \cdots & 0 & -a \\
            0 & \ddots & & & \udots & -(l-1)\theta \\
            \vdots & & 1+a & -a & \udots & 0 \\
            \vdots & & 1 & a-\frac{l}{2}\theta & & \vdots \\
            0 & \udots & \udots & \left(\frac{l}{2}-1\right)\theta & & \vdots \\
            0 & \udots & \udots &\ddots & \ddots & 0 \\
            1 & -\theta & 0 & \cdots & \theta & a \\
        \end{vmatrix} =
        \intertext{Note that the constants now are only in the left half of matrix.
        Now let us do the same operation with the matrix rows:}
        =\det
            \left(
                \begin{tabular}{c|c}
                    $\begin{matrix}
                        1 + a & & 0 \\
                        & \ddots & \\
                        0 & & 1 + a \\
                    \end{matrix}$
                    & 
                    *
                    \\ \hline
                    *
                    & 
                    $\begin{matrix}
                        2a-\frac{l}{2}\theta & \left(\frac{l}{2}+1\right)\theta & & 0\\
                        \left(\frac{l}{2}-1\right)\theta & 2a & \ddots & \\
                        & \ddots & \ddots & (l-1)\theta \\
                        0 & & \theta & 2a \\
                    \end{matrix}$
                    \\
                \end{tabular}
            \right).
            \intertext{Here the left-lower and the right-upper parts of the matrix 
            marked by $*$ do not contain constant terms. Therefore, the lowest homogeneous part of the function $f=\det(\mathcal{G}_l + r\operatorname{Id})$ defining 
            $\Xi_l^+$ is equal to
          the determinant of the lower right matrix part being a three-diagonal matrix.}
    \end{gather*}
    The determinant of the latter matrix is obviously not divisible by $\theta$, which follows from 
    (\ref{gnr}). Let us show that it is a 
    product of pairwise non-proportional linear forms. 
    
    Fix a $\theta>0$. We look for those $a$ for which the determinant vanishes; they 
    are equal to  minus half-eigenvaules  of the same matrix without the 
    $2a$ on the diagonal.  
    Since pairs of elements symmetric~w.\,r.\,t. the diagonal have positive 
    coefficients at~$\theta$, this matrix can be conjugated by a constant diagonal matrix
    in such a way that the result would be a symmetric three-diagonal matrix
    with positive entries below and above the diagonal.

    From~\cite[\S\,1.3.10]{ilyin} 
    the eigenvalues of such matrix are real and 
    do not coincide. This implies that for every given $\theta>0$ the above determinant 
    vanishes if and only if  $a$ takes one of $\frac l2$ distinct real values. 
    Therefore, the linear forms in its product decomposition are pairwise non-proportional. 

    So we conclude that for even~$l$ the equality~$\upmu_{p_+}(\overline{\Xi_l^+})=l/2$ holds,
    the point has $l/2$ tangent directions and~$\theta=0$ is not one of them. 
    
    Now consider the case of~$l$ odd. We process the same transformations with
    matrix (except the central row and column) and obtain
    \[
        \det
        \left(
            \begin{tabular}{c|c}
                $\begin{matrix}
                    1 + a &&& 0 \\
                    & \ddots && \\
                    && 1 + a & \\
                    0 &&& 2 + a \\
                \end{matrix}$
                & 
                *
                \\ \hline
                *
                & 
                $\begin{matrix}
                    2a & \frac{l+3}{2}\theta & & 0\\
                    \frac{l-3}{2}\theta & 2a & \ddots & \\
                    & \ddots & \ddots & (l-1)\theta \\
                    0 & & \theta & 2a \\
                \end{matrix}$
                \\
            \end{tabular}
        \right).
    \]
    In this case the lowest homogeneous component has 
    degree~$\frac{l-1}{2}=\left[\frac{l}{2}\right]=\upmu_{p_+}(\overline{\Xi_l^+})$.
    Analogously, all the eigenvalues are real and do not coincide.

    Now let us consider the point~$p_-$. Here we substitute~$r=a-1$.
    We make the same transformations, but we add columns and rows instead 
    of subtracting.

    In the case of~$l$ even we obtain the same three-diagonal matrix with
    $\theta$ replaced by~$-\theta$. All the rest of reasoning is the same.

    The case of~$l$ odd for~$p_-$ is slightly different, because the constants
    at the central element of matrix vanish:
    \[
        \det
        \left(
            \begin{tabular}{c|c}
                $\begin{matrix}
                    1 + a && 0 \\
                    & \ddots & \\
                    0 && 1 + a \\
                \end{matrix}$
                & 
                *
                \\ \hline
                *
                & 
                $\begin{matrix}
                    a & \frac{l+1}{2}\theta & & 0\\
                    \frac{l-1}{2}\theta & 2a & \ddots & \\
                    & \ddots & \ddots & (l-1)\theta \\
                    0 & & \theta & 2a \\
                \end{matrix}$
                \\
            \end{tabular}
        \right).
    \]
    In this case the lowest component has 
    degree~$l-\frac{l-1}{2}=l-\left[\frac{l}{2}\right]=\upmu_{p_-}(\overline{\Xi_l^+})$.
\end{proof}

\begin{proposition}
    \label{GeneraMax}
    The genus of~$\widehat{\Xi}_l^\pm$ does not exceed
    \[
        \begin{cases}
          \left(\dfrac{l-2}{2}\right)^2, & \text{ for $l$ even,} \\
            \\
            \left(\dfrac{l-1}{2}\right)\left(\dfrac{l-3}2\right), & \text{ for $l$ odd.}  
                    \end{cases}
    \]
    The equality takes place it and only if the curves $\widehat{\Xi}_l^{\pm}$ are non-singular, 
    which is equivalent to the statement that the complex affine curve 
    $\Gamma_l\subset\cc^2$ has no 
    singularities outside the line $\{\mu=0\}$. 
\end{proposition}

\begin{proof} Recall that the curves $\widehat{\Xi}_l^\pm$ lie in 
$\mathbb{P}^1\times\mathbb{P}^1$. We use the following formula for the 
geometric genus of an irreducible curve of bidegree 
$(d_1,d_2)$ in $\mathbb{P}^1\times\mathbb{P}^1$, 
see ~\cite[\S2,~Prop.\,3,~Ex\,2.2]{Serre}: 
\begin{equation}
        g(C) = (d_1-1)(d_2-1) - \delta,
    \label{fgd}\end{equation}
where~$\delta = \sum_P\delta(P)$ is the sum of so-called $\delta$-invariants
    of singular points of~$C$.  The analytic invariant $\delta(P)$ is positive
    if and only if the point~$P$ is singular, see~\cite[\S1,~Prop.\,1]{Serre}. The 
    $\delta$-invariant admits several definitions. Let us recall the geometric definition. 
    Let $(\alpha,P)$ be a germ of analytic curve, $P\in\cc^2_{x,y}$, 
    and let $f(x,y)=0$ be its equation. 
    Let $V$ be a small ball centered at $P$. Then $\alpha_\eps:=\{f(x,y)=\eps\}\cap V$, for $0<|\eps|\ll1$, is a smooth real two-dimensional closed surface with $r$ holes ({\it Milnor fiber}). The $\delta$-invariant
$\delta(P)$ topologically can be defined as the genus of the closed
surface obtained by attaching a sphere with $r$ holes to the surface $\alpha_\eps$.
See ~\cite{Milnor}. 

    Let us now apply formula (\ref{fgd}) to the curves 
    $\widehat\Xi_l^{\pm}\subset\mathbb P^1\times\mathbb P^1$, say, to the curve 
    with the sign $+$. 
    
    {\bf Claim 1.} {\it The bidegree of the curve $\widehat\Xi_l^{+}$ is equal to the vector 
    of  multiplicities $(\upmu_{p_+}(\overline\Xi_l^+), \upmu_{p_-}(\overline\Xi_l^+))$ 
    up to permutation.}
   
    \begin{proof}
    Blowing up each singularity $p_{\pm}$ separates the branches through it 
    and erases the 
    singularity of the blown up curve (i.e., its strict transform). 
    It yields a (-1)-curve $C_{\pm}$ intersecting the above strict transform  
    transversely in $\upmu_{p_\pm}$ distinct points. The latter intersection points 
    do not lie in the strict transform $C_{\infty}$ of the infinity line, since the branches at 
    $p_{\pm}$ (before the blow up) are transversal to the infinity line  (the last 
    statement of Proposition \ref{trpp}).   Therefore, 
    blowing down $C_{\infty}$ does not create additional singularities. 
    Moreover, the strict transforms of the above (-1)-curves $C_{\pm}$ 
     are vertical and 
    horizontal fibers in the product structure $\mathbb P^1\times\mathbb P^1$ of the 
    blown down surface. They intersect the curve $\widehat\Xi_l^{+}$ 
    transversely at $\upmu_{p_\pm}$ distinct smooth points, by the above argument. 
    This implies the statement of the claim.
   \end{proof}
   
   Substituting the formulas for the multiplicities from Proposition \ref{trpp} to the 
   bidegree vector in the  above claim and substituting everything to 
   genus formula (\ref{fgd}) yields the statement of Proposition \ref{GeneraMax}. 
   Its last statement follows from construction and the above blow-up argument. 
\end{proof}

\begin{corollary}
    The strict transforms of curves~$\overline{\Xi_l^\pm}$ for any~$l$ on~$X_7$ have no real
    singularities. Also, their projections~$\widehat{\Xi}_l^\pm\subset\mathbb{P}^1\times\mathbb{P}^1$
    are smooth over~$\mathbb{R}$.
\end{corollary}
The corollary follows from the blow up argument in the proof of Claim 1.

    One can see examples of curves~$\Xi_*^+$ on Fig.\,\ref{xi}

    \begin{figure}[h!]
        \centering
        \begin{subfigure}[$\Xi_2^+$]
            \centering
            \includegraphics[width=0.3\textwidth]{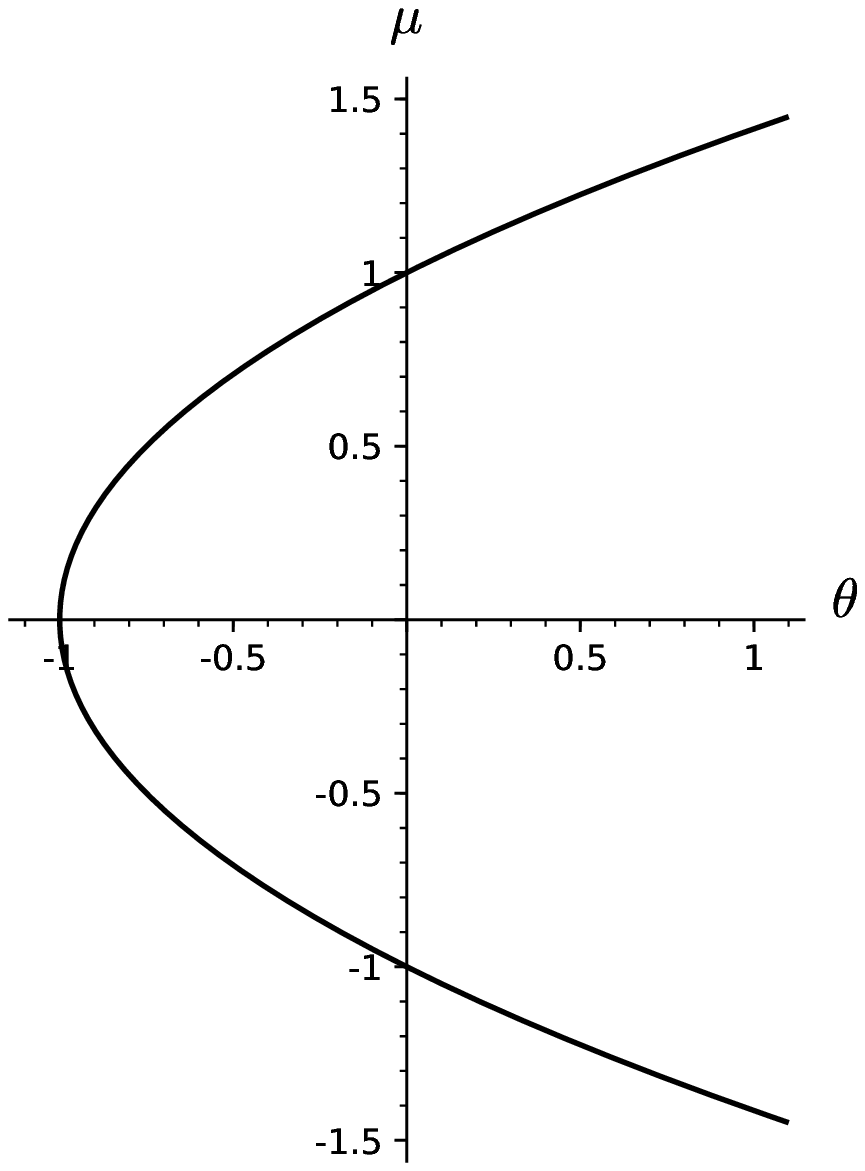}
        \end{subfigure}
        \begin{subfigure}[$\Xi_3^+$]
            \centering
            \includegraphics[width=0.3\textwidth]{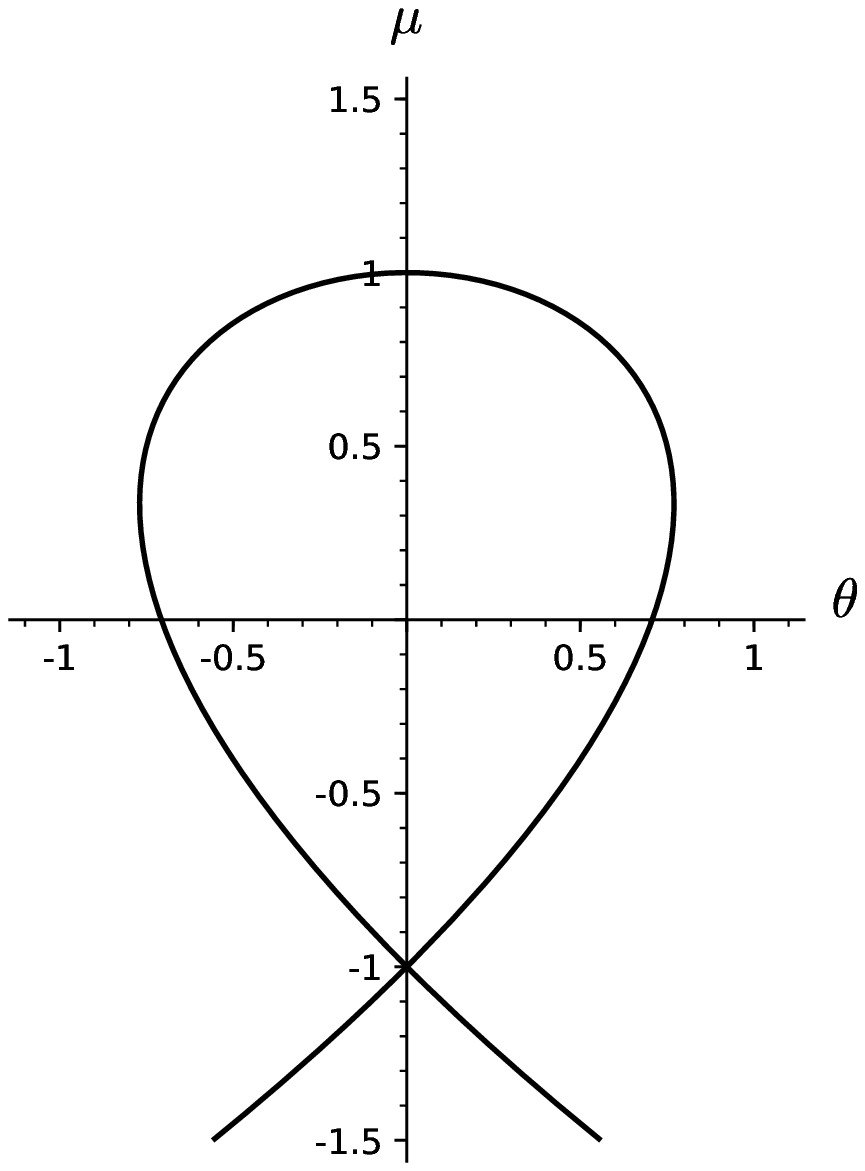}
        \end{subfigure}
        \begin{subfigure}[$\Xi_4^+$]
            \centering
            \includegraphics[width=0.3\textwidth]{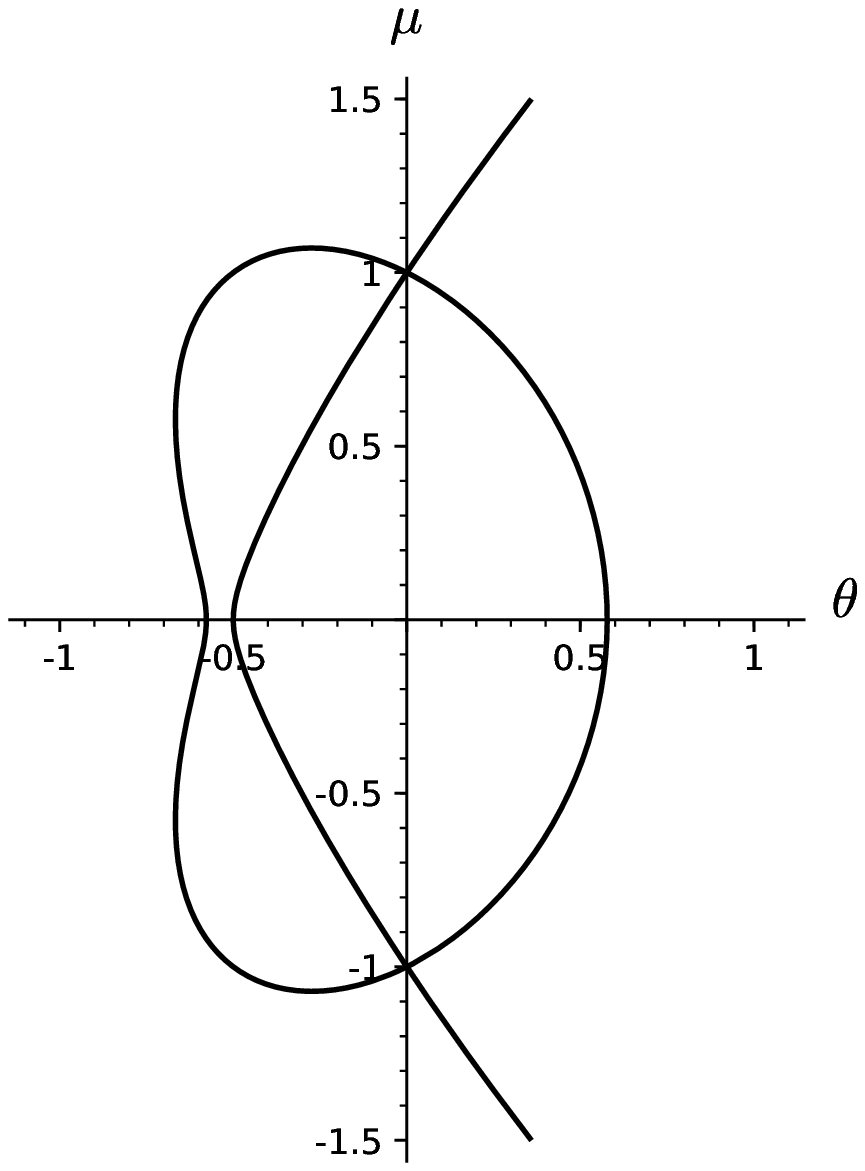}
        \end{subfigure}
        \begin{subfigure}[$\Xi_5^+$]
            \centering
            \includegraphics[width=0.3\textwidth]{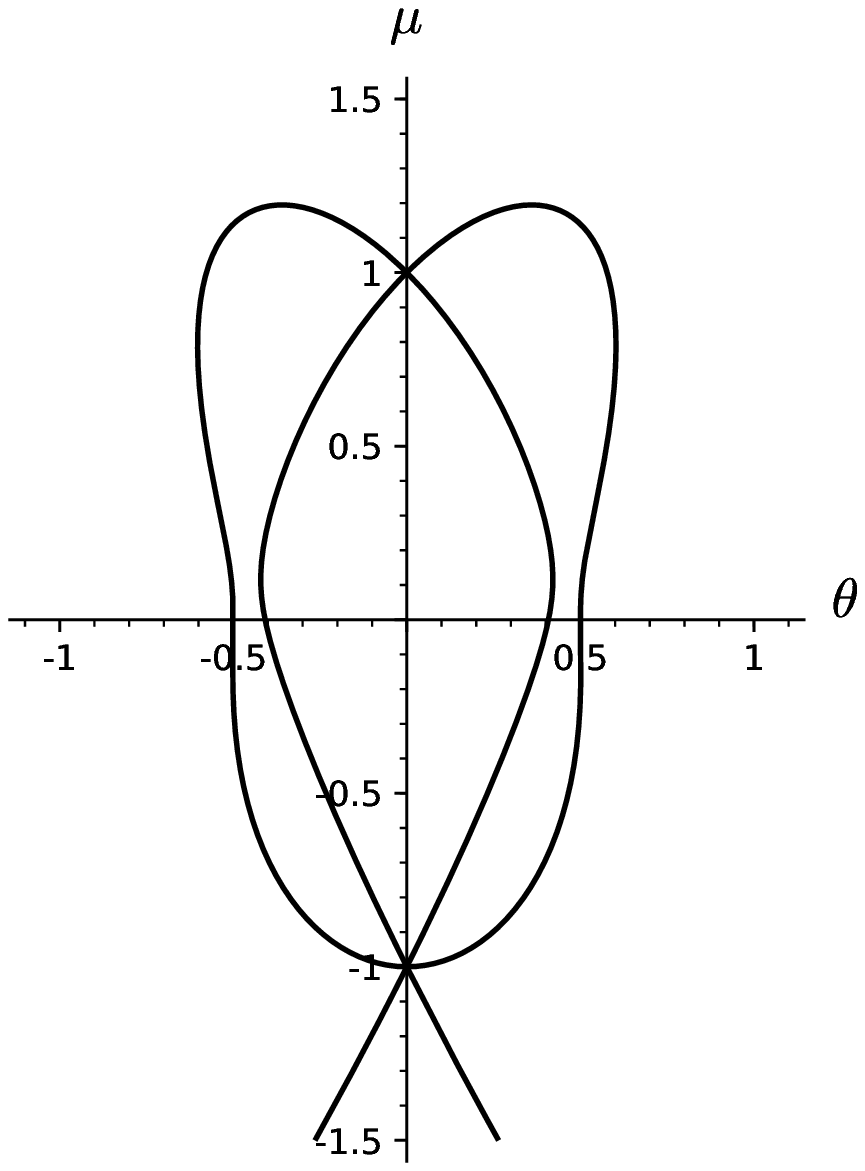}
        \end{subfigure}
        \begin{subfigure}[$\Xi_6^+$]
            \centering
            \includegraphics[width=0.3\textwidth]{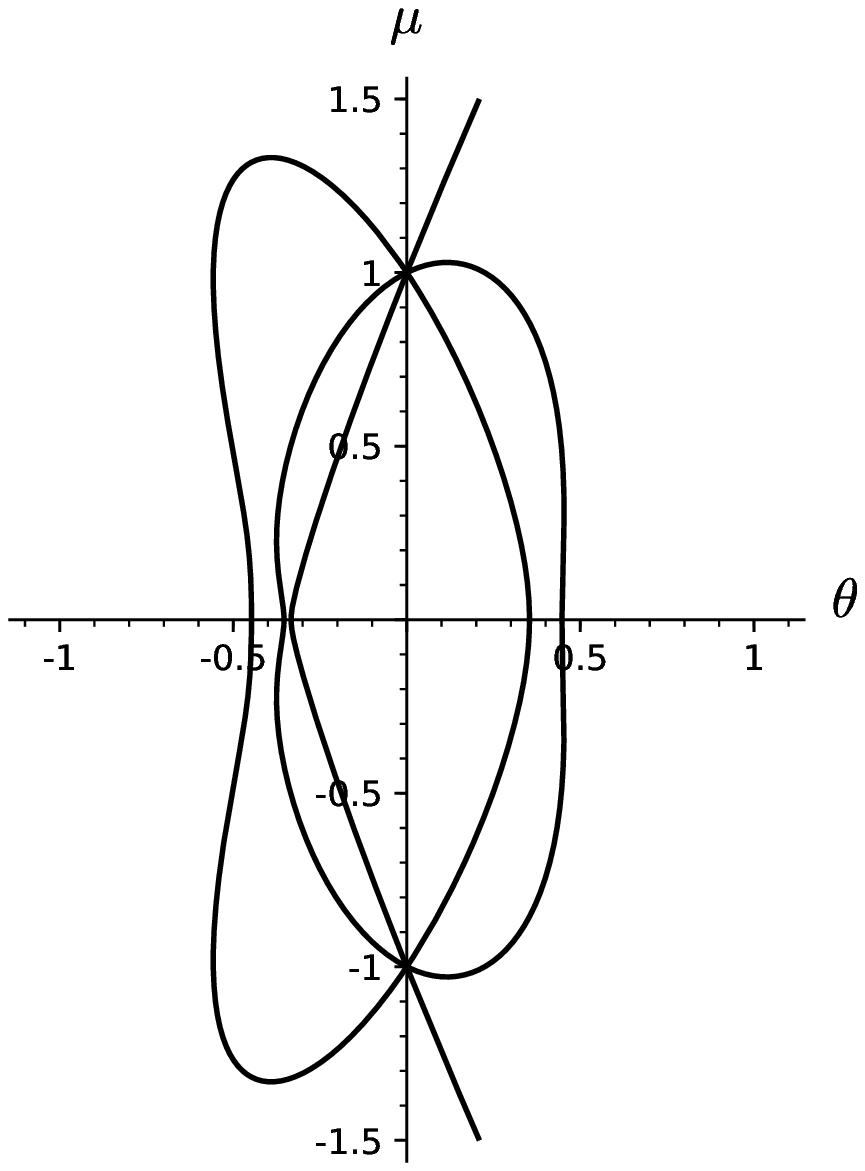}
        \end{subfigure}
        \begin{subfigure}[$\Xi_7^+$]
            \centering
            \includegraphics[width=0.3\textwidth]{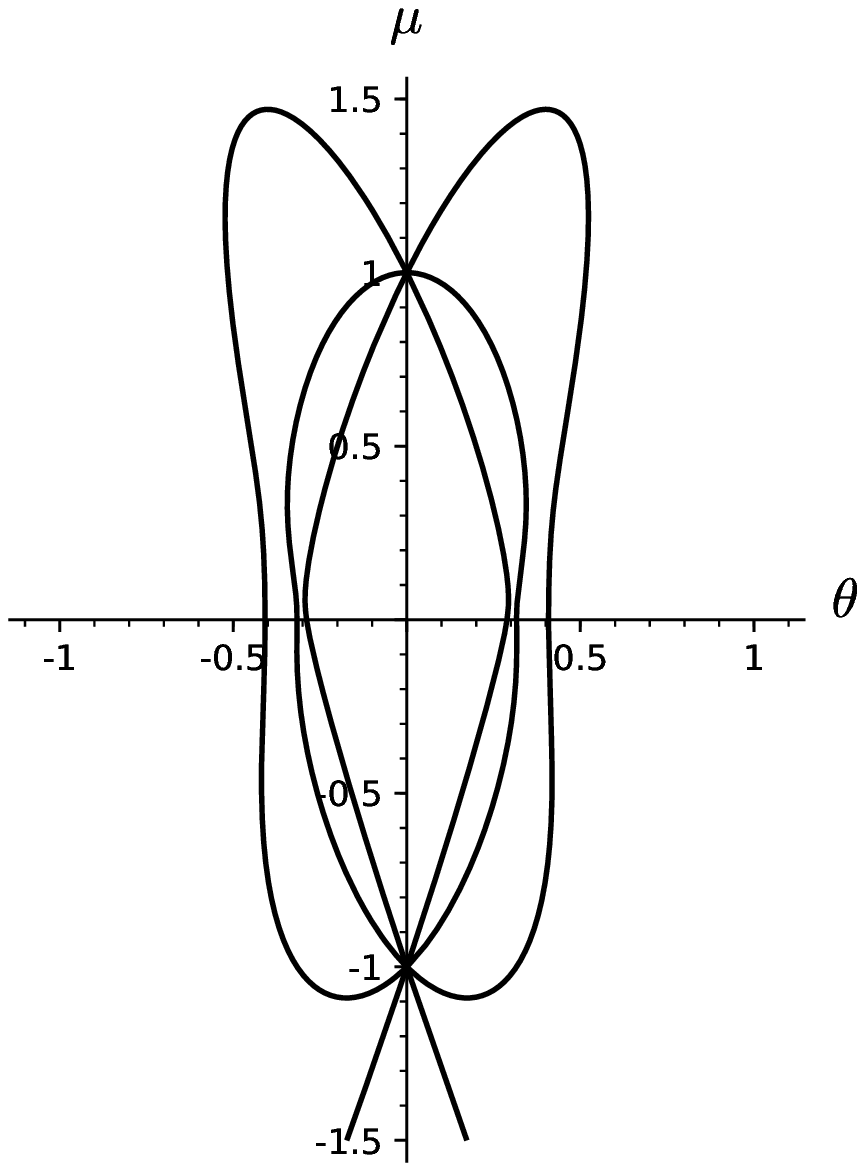}
        \end{subfigure}
        \caption{Some examples of curves~$\Xi_l^+$ in the affine chart~$r\ne0$.}
        \label{xi}
    \end{figure}

   \begin{corollary} Conjecture~\ref{congenu} on genus formula is equivalent to smoothness of some (any) of the 
    curves~$\widehat{\Xi}_l^\pm$ on~$\mathbb{P}^1\times\mathbb{P}^1$ over~$\mathbb{C}$ and also equivalent to the smoothness of the  complex curve 
    $\Gamma_l\setminus\{\mu=0\}\subset\cc^2$.\end{corollary}

\section{The Monotonicity Conjecture: proof for generalized simple intersections with 
smaller rotation numbers (A.A.Glutsyuk)}

Here we prove Theorem \ref{tmon}.  For its proof let us first recall that for every 
 $\mu>0$ the polynomial $P_l(\la,\mu^2)$ with fixed $\mu$ and variable $\la$ 
 is a polynomial of degree $l$ that has $l$ distinct real roots
 \begin{equation}\la_1<\la_2<\dots<\la_l; \ \ \ \la_j=\la_j(\mu),\label{lord}\end{equation}
 by Buchstaber--Tertychnyi Theorem \ref{tsimple} (see \cite[p.974, theorem 1]{bt0}). Set 
 $$R_j(\mu):=\la_j(\mu)+\mu^2.$$
 In the case, when $R_j(\mu)>0$, set 
 $$\omega_j(\mu):=\sqrt{\frac1{4R_j(\mu)}}>0, \ B_j(\mu):=l\omega_j(\mu), \ 
 A_j(\mu):=2\mu\omega_j(\mu),
$$
\begin{equation}\Pi_j(\mu):=(B_j(\mu),A_j(\mu)).\label{pij}
\end{equation}
Then the point $\Pi_j(\mu)$ is a generalized simple intersection for $\omega=\omega_j(\mu)$.

Our first goal is to show that $R_j(\mu)>0$ for every $\mu>0$ and to describe 
 the rotation numbers $s=s(j)$ of the generalized simple intersections defined by the roots $\la_j(\mu)$. 
 This is done in the three following propositions. 
\begin{proposition} 
For every $\mu>0$ one has $R_j(\mu)\neq0$; thus, each function 
$R_j(\mu)$ in $\mu>0$ has constant sign.
\end{proposition}
\begin{proof} If, to the contrary, $R_j(\mu)=0$ for some $\mu>0$, then 
$$0=P_l(\la_j(\mu),\mu^2)=P_l(\la_j,-\la_j)(\mu)=\det M_1(\la_j(\mu))=\la_j^l(\mu),$$
see the definition of the matrix function $M_1(\la)$ and formula (\ref{detm1}) 
for its determinant in Subsection 2.1. Finally, 
$$\la_j(\mu)=0, \ \mu^2=R_j(\mu)-\la_j(\mu)=0,$$
-- a contradiction. The proposition is proved.
\end{proof}
\begin{proposition} \label{sjc} Let $l\in\nn$. For every $j$ such that $R_j>0$ there exists 
 a unique pair $(s,\pm)$ of a number $s=s(j)\equiv l(\modulo 2)$, 
$0\leqslant s\leqslant l$, and a sign $\pm$  (for $s(j)=0$ the corresponding sign is always "$+$") 
such that for every $\mu>0$ the point $\Pi_j(\mu)$, see (\ref{pij}), is a generalized 
simple intersection of the axis $\La_l(\omega_j(\mu))$ with the boundary curve 
$\partial L_{s,\pm}(\omega_j(\mu))$.
\end{proposition}
\begin{proof} It follows by construction that each point $\Pi_j(\mu)$ is a 
generalized simple intersection 
depending continuously on $\mu>0$. The corresponding rotation number 
$s=s(j)(\mu)$ is integer-valued, by construction, and also continuous in $\mu$. 
Therefore, it is constant in $\mu>0$ and hence, depends only on $j$. 
Continuity and thus, constance of the sign $\pm=\pm(\mu)$ follows 
by the same argument and by well-definedness of the sign at 
a generalized simple intersection. Indeed, a generalized simple 
intersection with $\mu>0$ cannot lie in two distinct 
boundary components of a phase-lock 
area, since it is neither a growth point, nor a constriction. 
The proposition is proved.
\end{proof}
\begin{proposition} \label{prop7} Leg $l\in\nn$. 

1) For every $j=1,\dots,l$ one has $R_j>0$. 

2) The corresponding 
rotation number $s(j)$ from Proposition \ref{sjc} is given by the formula
\begin{equation}
s(j)=\begin{cases} l-j+1 & \text{ for odd } j,\\
l-j & \text{ for even } j.\end{cases}\label{sjform}
\end{equation}

3) For every $s$, $0<s<l$, choose arbitrary sign $\pm$. For $s=0$ 
choose the sign $+$. There exists a (unique) choice of sign $\pm=(\pm)_l$ for 
$s=l$ such that for every pair $(s,\pm)$ as above with 
$s\equiv l(\modulo2)$ and every $\mu>0$  there exists a 
unique $\omega>0$ for which $(l\omega, 2\mu\omega)$ is 
a generalized simple intersection of the axis $\Lambda_l(\omega)$ and the 
boundary curve $\partial L_{s,\pm}(\omega)$. The latter $\omega$ is equal to some 
$\omega_j(\mu)$. The above $j$ and $s$ are related by formula (\ref{sjform}); the number 
$j$ is uniquely determined by the pair $(s,\pm)$. 
\end{proposition}
\begin{proof} For every $s$, $0<s<l$,  every choice of sign $\pm$ 
and every $\omega>0$ small enough the 
boundary curve $\partial L_{s,\pm}(\omega)$ intersects $\La_l(\omega)$ at some point 
$\Pi^{s,\pm}(\omega)=(l\omega,A^{s,\pm}(\omega))$. A version of this statement holds for $s=0$ and the sign "$+$", and for $s=l$ 
and exactly one choice of sign. Both statements follow from Theorem  \ref{doubint} and Corollary \ref{cdoubint}. 
The number of  pairs $(s,\pm)$ with 
$s\equiv l(\modulo2)$  for which the above statement holds 
(including $s=0$ and $s=l$) 
is equal to  $l$, by Corollary \ref{cdoubint}. Set 
$$\mu^{s,\pm}:=\frac{A^{s,\pm}(\omega)}{2\omega}, \ \ \ s\equiv l(\modulo2).$$
Then $\Pi^{s,\pm}(\omega)=\Pi_j(\mu^{s,\pm})$ for some $j\in\{1,\dots,l\}$, and one has $R_j(\mu^{s,\pm})>0$. This 
follows from construction. The numbers $j=j(s,\pm)$ are uniquely defined by 
$(s,\pm)$, and they  are distinct for distinct pairs $(s,\pm)$. This follows from 
Proposition \ref{sjc}. The number of indices $j$ corresponding 
to  pairs $(s,\pm)$ is 
equal to $l$, as is the number of pairs $(s,\pm)$, and $R_j>0$ for 
these $j$. The total number of indices $j$ is equal to $l$. 
Therefore, the correspondence $j\mapsto(s,\pm)$ is one-to-one, 
 and  for $s=l$ the corresponding 
sign is uniquely determined, as in Theorem \ref{doubint}. This   
also implies that $R_j>0$ for all $j=1,\dots,l$. For every given $l\in\nn$, 
$\mu>0$ and $\omega>0$ the point $(B=l\omega,A=2\mu\omega)$ is a 
generalized simple intersection, if and only if it coincides with some of 
$\Pi_j(\mu)$, and then $\omega=\omega_j(\mu)$. This follows from 
construction. 

The above discussion proves   Proposition \ref{prop7}, except 
for formula (\ref{sjform}) and the last part of statement 3), which follows from this formula. 
Let us now prove (\ref{sjform}). 

{\bf Claim.} {\it The numbers $s(j)$ form a non-increasing sequence.} 

\begin{proof} Note that for every integer 
$s_1<s_2$ the phase-lock area $L_{s_1}$ lies on the left from the phase-lock area 
$L_{s_2}$. For every $s\in\nn$ consider the growth point $(\sqrt{s^2\omega^2+1},0)$ 
of the phase-lock area $L_s$, see \cite[corollary 3]{buch1}. In the renormalized 
coordinates  
$(l=\frac B{\omega}, \mu=\frac A{2\omega})$ its abscissa $\sqrt{s^2+\frac1{\omega^2}}$ 
moves from the left to the right, as $\omega$ decreases. Therefore, for every 
$l\in\nn$ and $s_1<s_2\leqslant l$, as $\omega>0$ decreases, 
 an intersection of the axis $\La_l(\omega)$ with the boundary $\partial L_{s_2}$ 
will appear earlier than its intersection with $\partial L_{s_1}$. The above intersections 
are generalized simple intersections, if $s_1\equiv s_2\equiv l(\modulo 2)$ and $s_2<l$. 
Indeed, they are either constrictions, or generalized simple intersections (by definition), and  
they cannot be 
constrictions by inequality $s_1<s_2<l$ and \cite[theorem 1.2]{4}. 

We already known that for given $l\in\nn$,  $\mu>0$ and $\omega>0$ a point 
$(B=l\omega, A=2\mu\omega)$ is a generalized simple intersection, if and only if 
$\omega=\omega_j(\mu)$ for some $j\in\{1,\dots,l\}$. One has 
$\omega_1(\mu)>\dots>\omega_l(\mu)$, by (\ref{lord}). This together with the above 
discussion implies that the sequence $s(j)$ of the corresponding rotation numbers 
does not increase and proves the claim.
\end{proof} 

The claim implies that the pairs $(s,\pm)(1),\dots,(s,\pm)(l)$ are ordered so that 
the pair $(l,(\pm)_l)$ goes first, then two pairs $((l-1),\pm)$ with both signs $\pm$, 
etc. This implies formula (\ref{sjform}) and proves Proposition \ref{prop7}.
\end{proof} 

\begin{proof} {\bf of Theorem \ref{tmon}.} Suppose the contrary: for  certain 
$k\in\nn$, $\mu_0>0$ and $j\in\{1,\dots,k\}$  a generalized simple intersection 
$z=\Pi_j(\mu_0)\in\La_k(\omega)\cap\partial L_{s,\pm}(\omega)$, 
 is a left-moving tangency; here $s=s(j)$, $\omega=\omega_j(\mu_0)$. 
We use the rescaled coordinates $(l=\frac B\omega, 
\mu=\frac A{2\omega})$ on $\rr^2$. 
Recall that we consider that $s<k$,  $s\equiv k(\modulo 2)$, 
by the conditions of Theorem \ref{tmon}  and 
since generalized simple intersections of a boundary curve $\partial L_{s,\pm}$ 
and an axis $\La_k$ may exist only for $s\in[0,k]$, $s\equiv k(\modulo 2)$, see 
Theorem \ref{bt02}. Then for $\omega>\omega_j(\mu_0)$ the boundary curve 
$\partial L_{s,\pm}(\omega)$ intersects the horizontal line $\mu=\mu_0$ on the right from 
the point $z=\Pi_j(\mu_0)$, by the definition of left-moving tangency (see Fig. 6). On the 
other hand, for every $\omega>0$ large enough the phase-lock area 
$L_s(\omega)$ should lie on the left from the line $\{ l=k\}$, by  \cite[proposition 3.4]{4} 
(which in its turn follows from Chaplygin's comparison theorem \cite{chap, buch1}). 
Hence, there exists a $\omega^*>\omega_j(\mu_0)$ such that the boundary curve 
$\partial L_{s,\pm}(\omega^*)$ contains $z$, by continuity. 
Then the point $z$ is a generalized simple intersection of the curve 
$\partial L_{s,\pm}(\omega^*)$ and the line $\La_k(\omega^*)$, and hence, 
$\omega^*=\omega_m(\mu_0)$ for a certain $m\in\{1,\dots,k\}$. One has $m<j$, since 
$\omega_m(\mu_0)=\omega^*>\omega_j(\mu_0)$, by construction and the sequence 
$\omega_n(\mu)$ decreases. Finally, two distinct indices $m$ and $j$ correspond to the 
same pair $(s,\pm)$, - a contradiction to the last statement of Proposition \ref{prop7}. This proves 
Theorem \ref{tmon}.
\end{proof}
 
\section{
    Some open problems and a corollary (A.A.Glutsyuk)
}

First we discuss some open problems on model of Josephson effect: on 
boundaries of phase-lock areas, constrictions and their complexifications. 
Afterwards we reformulate Theorem~\ref{irred2} in terms of family~\eqref{heun2*}
of special double confluent Heun equations and state an open question about them.

\subsection{
    Open problems related to geometry of phase-lock area portrait
}

To our opinion the main results of the paper and solutions of the next open 
problems might be useful in study of the geometry of the family of  phase-lock 
area portrait in the model of Josephson effect and dependence of 
the constrictions on the parameter $\omega$.

Let us formulate a complex version of Conjecture \ref{conexp} and related questions. 
To do this, for a given $l\in\zz$ consider the one-dimensional complex analytic 
subset 
\[
    C_l=\{(\mu,\omega)\in\cc^2\setminus\{\mu=0\} \ | \ 
    \text{ Riccati equation (\ref{ricc}) has trivial monodromy}\}.
\]
This is a complex analytic subset in 
$$\cc^2_{\mu\neq0}:=\cc^2_{\mu,\omega}\setminus\{\mu=0\}.$$ 
\begin{remark}
    Let $\omega>0$, $\mu\in\rr$, $l\in\zz$, $B=l\omega$, $A=2\mu\omega$. 
    Let $(B,A)$ be a constriction for the model of Josephson effect. 
    Then $(\mu,\omega)\in C_l$. Thus, the set $C_l$ can be viewed as the set 
    of {\it complex constrictions} in the complex axis 
    $\La_l=\{ B=\omega l\}\simeq\cc$. The numerical pictures (see, e.\,g., 
    Figs.~1,~2) show that the real constrictions viewed as points in $C_l$ 
    depend continuously on the parameter $\omega$. 
\end{remark}
\begin{proposition} For every $l\in\zz$ the  analytic subset $C_l$ 
has pure dimension 
one and is given by zeros of a global holomorphic function $\cc^2_{\mu\neq0}\to\cc$. 
For $l\geqslant0$  ($l<0$) 
the latter function is the Stokes multiplier $c_0$  (respectively, 
$c_1$)  of the irregular singular point 0 of linear system (\ref{tty}), see 
\cite[formula (2.2)]{g18} for the definition of the Stokes multipliers $c_{0,1}$. 
\end{proposition}
\begin{proof} It is known that for $l\in\zz$ triviality of the monodromy transformation 
of Riccati equation (\ref{ricc}) is equivalent to vanishing of both Stokes multipliers, 
$c_0=c_1=0$, see \cite[proof of  lemma 3.3]{4}. In the case, when $l\in\zz_{\geqslant0}$, 
this is equivalent to the existence of an entire solution of the double confluent Heun 
equation \eqref{heun2*}, see \cite[theorem 4.10]{bg2}. On the other hand, vanishing of just  one Stokes multiplier $c_0$ 
implies the existence of the latter solution.  See the proof of the real analogue of the latter  statement in \cite[proof of theorem 2.5]{g18}, which remains valid for complex parameters as well. The case, when $l\in\zz_{<0}$, is reduced to the case, when $l\in\zz_{\geqslant0}$, via linear isomorphism \cite[(2.14)]{g18} between solution spaces of system 
(\ref{tty}) and the same system with opposite sign at $l$. The latter isomorphism  conjugates the corresponding monodromies and 
inverses the numeration of the canonical basic solutions and thus, of the Stokes multipliers. 
Therefore, vanishing of the Stokes multiplier $c_0$ for system (\ref{tty}) with non-negative 
$l$ is equivalent to vanishing of the Stokes multiplier $c_1$ for the same system 
with opposite sign at $l$. This proves the proposition.
\end{proof}

{\bf Question 1.} 
How many irreducible components does the curve $C_l$ have?

{\bf Question 2.}
Let $\widehat C_l$ denote the normalization of the curve $C_l$: the Riemann 
surface bijectively parametrizing $C_l$ (except for self-intersections). Is it 
true that the projection 
\[
    \widehat C_l\to\cc, \ (\mu,\omega)\mapsto\omega
\]
is   a (ramified) covering over  $\cc$?
Describe its ramification points (if any).

\begin{conjecture}
    \label{conj2} 
    The above projection has no real critical points $(\mu,\omega)$.
\end{conjecture}

\begin{remark}
    A positive solution of Conjecture \ref{conj2} would imply the solution of 
    Conjecture~\ref{conexp}. This implication follows from~\cite[proposition~5.31]{bg2} 
    and the discussion before it. 
\end{remark}

Recall that $\Sigma_{\pm}$ is the locus of those parameter 
values for which the monodromy of Riccati equation (\ref{ricc}) fixes the 
point $\pm i$. The curve $C_l$  is included into the intersection 
$\Sigma_+\cap\Sigma_-$ via the rational mapping 
\[
    (\mu,\omega)\mapsto(l\omega, 2\mu\omega,\frac1{2\omega}).
\]
\begin{proposition} \label{moni} The surface $\Sigma_{\pm}$ consists exactly of those points, 
for which the monodromy of the 
Riccati equation (\ref{ricc}) is a parabolic (or idendical) M\"obius transformation fixing $\pm i$. Each point in $\cc^3_{(B,A,r)}$ 
corresponding to a parabolic (or identical) monodromy lies in the union 
$\Sigma=\Sigma_+\cup\Sigma_-$.
\end{proposition}
\begin{proof} The symmetry $(\Phi,z)\mapsto(-\Phi^{-1},z^{-1})$ of Riccati equation 
(\ref{ricc}) (see the proof of Proposition \ref{ppar}) restricted to the invariant 
fiber $\{ z=1\}$ conjugates the monodromy 
transformation of the fiber with its inverse. Its fixed  points 
are exactly $\pm i$. 
For every point in $\Sigma_{\pm}$ the monodromy and its inverse both fix 
the corresponding point $\pm i$. Therefore, their germs at the fixed point are analytically conjugated, and hence, 
are either parabolic, or identical. Vice versa, if the monodromy is 
parabolic, then its unique fixed point should be also fixed by the 
above symmetry and hence, should coincide with some of $\pm i$. 
The proposition is proved.
\end{proof} 

It could be useful to study analogues of Conjecture~\ref{conj2} 
for the projections of the surfaces $\Sigma_{\pm}$. 

{\bf Question 3.}
Is it true that the projections of the normalizations $\widehat\Sigma_{\pm}$ 
of the surfaces $\Sigma_{\pm}\subset\cc^3_{(B,A,r)}$ to $\cc^2_{(B,r)}$, 
$(B,A,r)\mapsto(B,r)$, are (ramified) coverings? 
Find their ramification loci (if any).

The surfaces $\hmcl_{\pm}$ were constructed as the minimal analytic subsets 
containing the union of families of $\pm$-boundary components of the phase-lock 
areas for the model of Josephson effect in $\rr^3_{(B,A,\frac1{2\omega})}$. 
It follows from definition that $\hmcl_{\pm}=\hmcl_{\pm}^{even}\cup\hmcl_{\pm}^{odd}
\subset\Sigma_{\pm}$, and all the surfaces under question are purely two-dimensional. 

\begin{remark}
    Theorem \ref{irred2} implies that $\hmcl_{\pm}^{even (odd)}$ are some 
    irreducible components of the surfaces $\Sigma_{\pm}$. Thus, {\it the 
    surface  $\Sigma=\Sigma_+\cup\Sigma_-$ consists of at least four irreducible 
    components.} {\it Each $\Sigma_{\pm}$ splits into two parts 
    $\Sigma_{\pm}^{even (odd)}$, which are unions of  some irreducible components, 
    by the index parity argument from the proof of Proposition \ref{ppar}.}
\end{remark}

{\bf Question 4.}
Is it true that $\hmcl_{\pm}=\Sigma_{\pm}$? In other terms, is it true 
that each  surface $\Sigma_{\pm}$ is a union of two irreducible components 
$\Sigma_{\pm}^{even (odd)}=\hmcl_{\pm}^{even (odd)}$?

\subsection{
    A corollary and an open problem on special double confluent Heun equation
}

Recall that the {\it monodromy operator} of linear equation~\eqref{heun2*} 
(or linear system~\eqref{tty}) is a linear operator acting on the space of its 
solutions by analytic extension along a counterclockwise circuit around the 
origin in the $z$-line. The monodromy operators of equation~\eqref{heun2*} and 
the corresponding linear system (\ref{tty}) are conjugated via the transformation 
\[
    E(z)\mapsto (u(z),v(z))=( 2i\omega z e^{-\mu z}(E'(z)-\mu E(z)),  
    e^{-\mu z}E(z)), \ \omega^2=\frac1{4(\la+\mu^2)}
\]
sending solutions of~\eqref{heun2*} to those of system~\eqref{tty}. The ratio 
\begin{equation}
    \Phi(z)=\frac{v(z)}{u(z)}=\frac{E(z)}{2i\omega z(E'(z)-\mu E(z))}
    \label{phiez}
\end{equation}
is a solution of the Riccati equation~\eqref{ricc}.

The transformation 
\[
    \Psi: (B,A,r=\frac1{2\omega})\mapsto(l=2Br, \ \la=r^2(1-A^2),\mu=Ar).
\]
sends the parameters of model of Josephson effect (or equivalently, the parameters 
of Riccati equation~\eqref{ricc}) to the parameters $(l,\la,\mu)$ 
of the corresponding special double confluent Heun equation~\eqref{heun2*}.

\begin{remark} \label{rmoni}
    A point $(B,A,r)\in\cc^3$ lies in $\Sigma=\Sigma_+\cup\Sigma_-$, if and only 
    if the corresponding Heun equation (\ref{heun2*}) has monodromy operator 
    with multiple eigenvalue. This follows from Proposition \ref{moni}. 
    If in this case the monodromy operator is a 
    Jordan cell, then its eigenfunction $E(z)$ (unique up to constant factor) 
    defines a meromorphic solution $\Phi(z)$, see~\eqref{phiez}, of Riccati 
    equation~\eqref{ricc}, for which $\Phi(1)=\pm i$. If we change $\omega$ 
    by $-\omega$ (which does not change the parameters $(l,\la,\mu)$ of Heun 
    equation), then the corresponding value $\Phi(1)=\pm i$ will change the sign. 
\end{remark}

Theorem~\ref{irred2}, Proposition \ref{moni} and  Question~$4$ 
can be equivalently reformulated as follows in terms of double confluent Heun 
equations~\eqref{heun2*}.

\begin{theorem}
    \label{irred3}
    Let $\widetilde\Sigma\subset\cc^3_{(l,\la,\mu)}$ denote the subset of those 
    parameters $(l,\la,\mu)\in\cc^3$ for which the corresponding special double 
    confluent Heun equation~\eqref{heun2*} has a monodromy operator with multiple 
    eigenvalue. The set of real points of the surface $\widetilde\Sigma$ 
    satisfying the inequality $\la+\mu^2>0$ (which thus come from boundaries of 
    the phase-lock areas) is contained in the union of two irreducible components 
    $\wt\mcl^{even (odd)}$ of the surface $\widetilde\Sigma$. 
\end{theorem} 

{\bf Addendum to Theorem~\ref{irred3}.}
{\it 
    In each of the above  components $\wt\mcl^{even (odd)}$ there exists an open and dense subset 
    of those points for which 

    \begin{itemize}
        \item the monodromy operator of equation~\eqref{heun2*} is a Jordan cell; 

        \item the monodromy eigenfunction $E(z)$ defines a solution~\eqref{phiez} 
            of Riccati equation~\eqref{ricc} meromorphic on $\cc^*$ that has 
            neither zeros, nor poles on the unit circle $S^1=\{|z|=1\}$ and has 
            even (respectively, odd) index along $S^1$.
    \end{itemize}
}

\begin{proof}
    The minimal analytic subset containing the real points under question 
    coincides with  $\Psi(\cup_{\pm}\hmcl_{\pm})$, which follows from definition. 
    The mapping $\Psi$ is polynomial of degree two, and its ramification 
    locus is the surface $\la+\mu^2=0$. The surface $\Sigma$ is the preimage 
    of the surface $\wt\Sigma$, by definition, Proposition \ref{moni} 
    and Remark \ref{rmoni}. The sign change 
    $\omega\mapsto-\omega$ permutes the sets $\Sigma_{\pm}$ and the sets 
    $\hmcl_{\pm}$. This implies that $\Psi$ sends the four irreducible components 
    $\hmcl_{\pm}^{even (odd)}$ to two components $\wt\mcl^{even (odd)}$. 
    The statements of the addendum follow from construction. 
\end{proof}

{\bf Question 5.}
Is it true that the whole surface $\wt\Sigma\subset\cc^3_{(l,\la,\mu)}$ 
is the union of the two above irreducible  components $\wt\mcl^{even (odd)}$?

\section{Acknowledgments}

We are grateful to V.M.Buchstaber for attracting our attention to Heun equations and the spectral curves. We are grateful to him for  helpful discussions and remarks. We are also grateful to S.K.Lando for 
informing us about Vinnikov's paper \cite{vin}.

\end{document}